\documentclass[oneside,american]{amsart}
\usepackage[T1]{fontenc}
\usepackage[latin9]{inputenc}
\usepackage{textcomp}
\usepackage{mathrsfs}
\usepackage{amstext}
\usepackage{amsthm}
\usepackage{amssymb}
\usepackage{enumerate}
\usepackage{cancel}
\makeatletter
\numberwithin{equation}{section}
\numberwithin{figure}{section}
\theoremstyle{plain}
\newtheorem*{thm*}{\protect\theoremname}
\theoremstyle{plain}
\newtheorem{thm}{\protect\theoremname}
\theoremstyle{plain}
\newtheorem{lem}[thm]{\protect\lemmaname}
\theoremstyle{plain}
\newtheorem{cor}[thm]{\protect\corollaryname}
\theoremstyle{plain}
\newtheorem{prop}[thm]{\protect\propositionname}
\theoremstyle{definition}
\newtheorem{defn}[thm]{\protect\definitionname}
\theoremstyle{definition}
\newtheorem{example}[thm]{\protect\examplename}
\theoremstyle{remark}
\newtheorem{rem}[thm]{\protect\remarkname}
\theoremstyle{remark}
\newtheorem*{rem*}{\protect\remarkname}
\theoremstyle{plain}
\newtheorem*{prop*}{\protect\propositionname}

\usepackage[svgnames,dvipsnames]{xcolor}
\newcommand{\jr}[1]{{\color{red} #1 }}

\usepackage{soul}
\def\GL{\mathrm{GL}}
\usepackage[all]{xy}

\makeatother

\usepackage{babel}
\providecommand{\corollaryname}{Corollary}
\providecommand{\definitionname}{Definition}
\providecommand{\examplename}{Example}
\providecommand{\lemmaname}{Lemma}
\providecommand{\propositionname}{Proposition}
\providecommand{\remarkname}{Remark}
\providecommand{\theoremname}{Theorem}

\usepackage[colorlinks=true, linkcolor=blue, 
hyperfootnotes=true,citecolor=blue,urlcolor=black]{hyperref}

\makeatletter
\@namedef{subjclassname@2020}{%
  \textup{2020} Mathematics Subject Classification}
\makeatother

\begin{document}
\title[Hypertranscendence and $q$-difference equations]{Hypertranscendence and $q$-difference equations over elliptic function
fields}
\author{Ehud de Shalit} 
\address{Einstein Institute of Mathematics, The Hebrew University of Jerusalem.}
\email{ehud.deshalit@mail.huji.ac.il}
\author{Charlotte Hardouin}
\address{Institut de Math\'ematiques de Toulouse, Universit\'e Paul Sabatier, 118, route de Narbonne, 31062 Toulouse, France.}
\email{hardouin@math.univ-toulouse.fr}
\author{Julien Roques}
\address{Universite Claude Bernard Lyon 1, CNRS, Ecole Centrale de Lyon, INSA Lyon, Universit\'e Jean Monnet, ICJ UMR5208, 69622 Villeurbanne, France.}
\email{Julien.Roques@univ-lyon1.fr}
\subjclass[2020]{39A06, 12H05}
\keywords{Difference equations, differential Galois theory, elliptic functions, hypertranscendence.}
\begin{abstract}
The differential nature of solutions of linear difference equations over the projective line was recently elucidated. In contrast, little is known about the differential nature of solutions of linear difference equations over elliptic curves. In the present paper, we study power series $f(z)$ with complex coefficients satisfying a linear difference equation over a field of elliptic functions $K$,
with respect to the difference operator $\phi f(z)=f(qz)$, $2\le q\in\mathbb{Z}$,
arising from an endomorphism of the elliptic curve. Our main theorem
says that such an $f$ satisfies, in addition, a polynomial differential
equation with coefficients from $K,$ if and only if it belongs to
the ring $S=K[z,z^{-1},\zeta(z,\Lambda)]$ generated over $K$ by
$z,z^{-1}$ and the Weierstrass $\zeta$-function. This is the first elliptic extension of recent theorems of Adamczewski, Dreyfus and Hardouin  concerning the differential transcendence of solutions of difference equations with coefficients in $\mathbb{C}(z),$ in which various difference operators were considered (shifts, $q$-difference
operators or Mahler operators). While the general approach, of using
parametrized Picard-Vessiot theory, is similar, many features, and
in particular the emergence of monodromy considerations and the ring
$S$, are unique to the elliptic case and are responsible for non-trivial difficulties. We emphasize that, among the intermediate results,  
we prove an integrability result for difference-differential systems over elliptic
curves which is a genus one analogue of the integrability results obtained by Sch\"afke and Singer over the projective line.
\end{abstract}

\maketitle

\section{Introduction}

In 1886 Otto H\"{o}lder \cite{Hol} proved that the Gamma function $\Gamma(z)$
is not only transcendental, but \emph{hypertranscendental}: it does
not satisfy any polynomial differential equation $P(f,f',,...,f^{(r)})=0$
whose coefficients are polynomials in $z$. This lies in contrast
to other well-known transcendental functions like $\exp(z),$ hypergeometric
functions or theta functions, which all satisfy familiar (and important)
polynomial differential equations over the field $\mathbb{C}(z)$
of rational functions. 

It turns out that what prevents $\Gamma(z)$ from satisfying a polynomial
differential equation over $\mathbb{C}(z)$ is the relation
\[
\Gamma(z+1)-z\Gamma(z)=0,
\]
in itself an instance of a linear \emph{difference equation} in the
difference operator $\phi f(z)=f(z+1).$ Indeed, as Adamczewski, Dreyfus
and Hardouin showed recently \cite{A-D-H}, if $F$ is a field of
meromorphic functions on $\mathbb{C}$, invariant under $\phi,$ satisfying
the following two properties
\begin{itemize}
\item $F^{\phi}=\{f\in F|\,\phi f=f\}=\mathbb{C},$
\item $F\cap\mathbb{C}(z,\exp(az)|\,a\in\mathbb{C})=\mathbb{C}(z)$,
\end{itemize}
and if $f\in F$ satisfies a linear $\phi$-difference equation \emph{and}
a polynomial differential equation, both with coefficients from $\mathbb{C}(z)$,
then $f\in\mathbb{C}(z)$. This includes H\"{o}lder's theorem as a special
case. The above-cited paper considered several other theorems of the
same nature, pertaining to difference operators that are shifts, $q$-difference
operators or Mahler operators, all over the ground field $\mathbb{C}(z).$
This gives us a good understanding of the differential nature of solutions of difference equations on $\mathbb{P}^{1}(\mathbb{C})$.

By Hurwitz's automorphisms theorem, any automorphism of a compact Riemann surface of genus $g >1$   is of  finite order. Therefore,  the algebro-differential study of solutions of difference equations over Riemann surfaces reduces to the genus zero or one case. Indeed,  the structure of a difference equation associated with a finite order automorphism is not rich enough to capture the algebraic nature of its solutions.  For genus one Riemann surfaces, that is, complex elliptic curves, there are essentially two kinds of endomorphisms: translations by a point of the curve and isogenies.

Our goal in the present work is to consider elliptic function fields
as ground fields, and a difference operator arising from an isogeny
of the elliptic curve\footnote{The case of a translation on the elliptic curve will be addressed in a forthcoming
paper by the last two authors.}. In this framework, we prove  an  analogue of the main result of \cite{A-D-H}. In our case, the dichotomy between rational solutions (that is, those belonging to the function field of the curve) and differentially transcendental solutions, observed in the genus zero case, is no longer true. Overcoming this difficulty leads to an explicit description of all the differentially algebraic solutions in terms of special functions. To describe our main result, let us fix some notation.

\bigskip{}

Let $\Lambda_{0}$ be a lattice in $\mathbb{C}$ and $K$ the field
of all meromorphic functions in the complex plane that are $\Lambda$-periodic
with respect to some sublattice $\Lambda\subset\Lambda_{0}.$ Fix
$2\le q\in\mathbb{Z}$. The field $K$ admits an \emph{automorphism
$\phi$ }and a\emph{ derivation $\partial$ }given by\emph{
\[
\phi f(z)=f(qz),\,\,\,\,\,\partial(f)=f',\,\,\,\,\partial\circ\phi=q\phi\circ\partial.
\]
}Both $\phi$ and $\partial$ extend to $M,$ the field of all meromorphic
functions in the complex plane, and to the field of Laurent power
series $F=\mathbb{C}((z)).$ Clearly, $K\subset M\subset F.$

A power series $f\in F$ is said to satisfy a linear homogeneous $\phi$\emph{-difference
equation} (called also a $q$-difference equation) over $K$ if
\begin{equation}\label{eq:phi diff eq intro}
\sum_{i=0}^{n}a_{n-i}\phi^{i}f=\sum_{i=0}^{n}a_{n-i}(z)f(q^{i}z)=0 
\end{equation}
for some $n$ and $a_{j}\in K.$ It is called \emph{$\partial$-algebraic
}over $K$ if there exists a nonzero polynomial $P\in K[X_{0},...,X_{r}]$,
for some $r\ge0,$ such that
\[
P(f,\partial f,...,\partial^{r}f)=0.
\]
A power series that is not $\partial$-algebraic is called \emph{hypertranscendental
(over $K$).}

The main result of this paper is the following complete description of the $\partial$-algebraic solutions in $F$ of $\phi$-difference equations of the form \eqref{eq:phi diff eq intro}.

\begin{thm*}[Theorem \ref{thm:Main theorem}]
Assume that $f\in F$ satisfies a non-trivial linear homogeneous $\phi$-difference
equation over $K.$ Then, the following properties are equivalent: 
\begin{enumerate}
\item $f$ is $\partial$-algebraic over $K$;
\item $f$ lies in the ring generated over $K$ by $z^{\pm1}$ and the
Weierstrass zeta function $\zeta(z,\Lambda_{0})$ of $\Lambda_{0}$, {\it i.e.},
\[
f\in S=K[z^{-1},z,\zeta(z,\Lambda_{0})].
\]
\end{enumerate}
\end{thm*}
The ring $S$ seems ubiquitous when studying functional equations over elliptic curves. It appeared also in \cite{dS21}, in the
classification of elliptic $(p,q)$-difference modules (for $p$ and
$q$ relatively prime $\ge2$). Its emergence in that work was attributed
to the appearance of certain non-trivial vector bundles (Atiyah's
bundles) over the elliptic curve $\mathbb{C}/\Lambda_{0}$. In the
present work, the same ring $S$, or rather its subring $S_{0}=K[z,\zeta(z,\Lambda_{0})]$,
arises, in the context of \emph{linear} differential equations over $K,$ from \emph{monodromy
considerations} (see Theorem \ref{thm:Entries of U are in S_0}).
Once we pass to arbitrary \emph{polynomial} differential equations over $K$, the ring $S$
arises as the \emph{Picard-Vessiot ring} of a certain fundamental
$\partial$-integrable $\phi$-module. Some properties of the ring
$S$ that are needed in later proofs are explored in \S2 and Appendix
A.
\bigskip{}

As an intermediate step toward the proof of the above theorem, we prove the following integrability result describing the elements in $F$ that satisfy both a linear homogeneous $\phi$-difference
equation, and a \emph{linear homogeneous ordinary differential equation} over $K$. It is the first elliptic analogue of several integrability results obtained for difference equations over the projective line, see \cite{ramis1992growth,bezivin1993solutions,bezivin1994classe,SS16}.

\begin{thm*}[Corollary \ref{cor:Little theorem-1}]
 Assume that $f\in F$ satisfies a linear homogeneous $\phi$-difference
equation over $K.$ Then, the following properties are equivalent: 
\begin{enumerate}
\item $f$ satisfies a linear homogeneous ordinary differential equation over $K$;
\item $f$ lies in the ring $S_{0}$ generated over $K$ by $z$ and the
Weierstrass zeta function $\zeta(z,\Lambda_{0})$ of $\Lambda_{0}$, {\it i.e.},
\[
f\in S_{0}=K[z,\zeta(z,\Lambda_{0})].
\]
\end{enumerate}
\end{thm*}

As implied above,
the proof of this result uses monodromy considerations for modules
with a connection over $K,$ which admit, in addition, a $\phi$-structure
(termed \emph{$\phi$-isomonodromic}). Dually, these modules can be
regarded as $\partial$-integrable $\phi$-modules. Section \ref{sec:equations and modules}
reviews the language of $\phi$-, $\partial$- and $(\phi,\partial)$-modules,
its relation to linear systems of $\phi$-difference and differential
equations, and the important concepts of $\phi$-isomonodromy and
$\partial$-integrability. Sections 6,7 and 8 lead the way to Theorem
\ref{thm:Entries of U are in S_0} and Corollary \ref{cor:Little theorem-1},
but the results obtained on the way, like the relation between integrability
and solvability (Corollary \ref{cor: integrability implies solvability-1})
are of independent interest, and are used again in the proof of the
Main Theorem.

\bigskip{}

Once we consider arbitrary polynomial differential equations over $K$, we are forced
to use the machinery of $\delta$-parametrized Picard Vessiot (PPV)
theory of $\phi$-difference equations, and $\delta$-parametrized
$\phi$-difference Galois theory. This theory, expounded in \cite{H-S08},
but less familiar than classical (non-parametrized) Picard Vessiot
theory, is used in our work in roughly the same manner as in \cite{A-D-H},
where the ground field is $\mathbb{C}(z).$ The details pertaining
to the ground field are, of course, different. In particular, since
PPV theory only works under the assumption that $\phi$ and $\delta$
commute, we must replace the derivation $\partial$ by $\delta=z\partial.$
This derivation, however, exists only over the field $K'=K(z),$ forcing
us to base-change from $K$ to $K'$, and then use descent arguments.
Classical and $\delta$-parametrized Picard Vessiot theory are surveyed
in Section \ref{sec:Picard-Vessiot-and-parametrized}. Section \ref{sec:A-Galoisian-criterion}
connects the PPV theory to the framework of $\phi$-modules, providing
a Galoisian criterion for $\delta$-integrability (a variant of results
already to be found in the literature).

The proof of the Main Theorem occupies Sections 9-12. Similarly to
the program carried out in \cite{A-D-H} over $\mathbb{C}(z),$ it
starts with the rank 1 case, where again, some elliptic function theory
is needed. We then consider an extreme case, in which the difference
Galois group $G$ associated with our $\phi$-module is \emph{simple.
}A deep theorem of Cassidy (Theorem 19 of \cite{Cas89}, see also
Theorem \ref{thm: Cassidy's theorem on simple Zariski closure} below)
allows us to deduce that in this case, the fundamental matrix of solutions
has a maximal $\delta$-transcendence degree, and in particular any
particular solution is hypertranscendental. This also settles the
Main Theorem in the \emph{irreducible case,} \emph{i.e.}, when the $\phi$-module
associated to the difference equation satisfied by $f$, is irreducible,
or, equivalently, when the standard representation of $G$ is irreducible.

To treat the general case, we use an inductive approach, similar to
the one used in \cite{A-D-H}. The rank 1 case, proved right at the
beginning, becomes instrumental. The last stage of the proof may be
described as ``Galois acrobatics''. At  the very final step, Proposition \ref{prop:technical_elliptic}, proved by a  technical tour de force and unique
to the elliptic set-up, plays a crucial role. 

\bigskip{}

We end with a word for the experts, regarding our use of $\delta$-parametrized
Picard Vessiot theory, and linear differential algebraic groups (LDAG's)
in general. We work in the classical language of Weil and Kolchin,
over a ``universal'' differentially closed field of constants $\mathbb{C}\subset\widetilde{C}$,
requiring descent arguments to come down to $K$ and $S$. While a
scheme-theoretic or Tannakian approach has been developed by various
authors to some extent, not all the results we need are in the literature,
and estabilishing them would have taken us beyond the scope of this
work.

\vskip 5 pt
\noindent {\bf Acknowledgements.} The work of the second and third authors was supported by the ANR De rerum natura project, grant ANR-19-CE40-0018 of the French Agence Nationale de la Recherche.

\section{\label{sec:Elliptic-functions}Elliptic functions and related rings}
\subsection{\label{subsec:K phi and partial}The ground field $K$, the automorphism
$\phi$ and the derivation $\partial$}

In this paper, we use standard notation of difference and differential algebra which can be found in \cite{Kol}, \cite{Cohn}. Algebraic attributes (e.g. Noetherian)
are understood to apply to the underlying ring. Attributes that apply to the difference (resp. differential) structure are usually prefixed with $\phi$ (resp. $\partial$). For instance, a $\phi$-ring is a ring with an endomorphism $\phi$, a $\phi$-ideal is an ideal of a $\phi$-ring that is set-wise invariant by $\phi$ etc.
\subsubsection{The field $K$}

For a lattice $\Lambda\subset\mathbb{C}$ we denote by $K_{\Lambda}$
the field of $\Lambda$-periodic meromorphic functions ($\Lambda$-elliptic
functions). It is well known that
\[
K_{\Lambda}=\mathbb{C}(\wp(z,\Lambda),\wp'(z,\Lambda))
\]
is generated over $\mathbb{C}$ by the Weierstrass $\wp$-function
\[
\wp(z,\Lambda)=\frac{1}{z^{2}}+\sum_{0\ne\omega\in\Lambda}\left(\frac{1}{(z-\omega)^{2}}-\frac{1}{\omega^{2}}\right)
\]
and its derivative. If $\Lambda'\subset\Lambda$ is a sublattice of
$\Lambda$ then $K_{\Lambda}\subset K_{\Lambda'}$. Two lattices $\Lambda$
and $\Lambda'$ are called \emph{commensurable} if $\Lambda\cap\Lambda'$
is a lattice, necessarily of finite index in each of them. Equivalently,
$\Lambda$ and $\Lambda'$ are commensurable if their $\mathbb{Q}$\jr{-}spans coincide: $\mathbb{Q}\Lambda=\mathbb{Q}\Lambda'$. The notion
of being commensurable is an equivalence relation on the set of lattices.
Fix an equivalence class $\mathfrak{L}$ (called a commensurability
class) and let
\[
K=\bigcup_{\Lambda\in\mathfrak{L}}K_{\Lambda}.
\]
It is readily seen that $K$ is a field, indeed equal to the union
of $K_{\Lambda}$ for all sublattices $\Lambda\subset\Lambda_{0}$,
if $\Lambda_{0}$ is any given member of the class $\mathfrak{L}.$
If $E_{\Lambda}$ is the complex elliptic curve whose associated Riemann
surface is $\mathbb{C}/\Lambda$, then $K_{\Lambda_{0}}$ is the field
of rational functions on $E_{\Lambda_{0}},$ and $K$ is the field
of rational functions on its universal cover (in the algebraic sense)
$\lim_{\leftarrow}E_{\Lambda},$ where the limit ranges over all the
unramified coverings $E_{\Lambda}\to E_{\Lambda_{0}}$ ($\Lambda\subset\Lambda_{0}).$

\subsubsection{The automorphism $\phi$}

Let
\[
\phi(z)=qz
\]
for $q\in\mathbb{Z}_{\ge 2}.$ As $\phi$ preserves every $\Lambda\in\mathfrak{L},$
it induces a non-trivial endomorphism (an isogeny) of $E_{\Lambda}$,
and an endomorphism $\phi$  of $K_{\Lambda}$  such that 
$(\phi f)(z)=f(qz)$ for any $f$ in $K_{\Lambda}$. As $\phi(K_{q\Lambda})=K_{\Lambda}\subset K_{q\Lambda},$
we see that $\phi$ induces an \emph{automorphism} of the field $K$.
In the language of difference algebra, $\phi$ is a \emph{difference
operator}, and the pair $(K,\phi)$ is a difference field. 

\begin{rem}
If the lattices $\Lambda\in\mathfrak{L}$ admit complex multiplication
in a quadratic imaginary field $k$, then $q$ could be taken to be
any non-zero non-unit element of the ring of integers $\mathcal{O}_{k}$.
For simplicity, however, we assume that $q$ is a rational integer. 
\end{rem}

The following result will be used much later.
\begin{lem}
\label{lem:no finite phi extensions of K}The field $K$ does not
admit any finite field extension $L/K$ to which $\phi$ extends as
an automorphism.
\end{lem}

\begin{proof}
See Proposition 7(iii) in \cite{dS23}.
\end{proof}

\subsubsection{The field of Laurent series $F$}

Associating to an elliptic function $f\in K$ its Taylor-Maclaurin
expansion at $0,$ the field $K$ embeds in the field of Laurent series
\[
F=\mathbb{C}((z)).
\]
Clearly $\phi$ induces an automorphism of $F$ as well, compatible
with the embedding $K\hookrightarrow F.$ As the \emph{field of} $\phi$\emph{-constants}
\[
F^{\phi}:=\{f \in F \ \vert \ \phi f=f \}=\mathbb{C},
\]
a-fortiori $K^{\phi}:=\{f \in K \ \vert \ \phi f=f \}=\mathbb{C}.$

\subsubsection{The derivation $\partial$}

The fields $K$ and $F$ are equipped with the derivation
\[
\partial=\frac{d}{dz}.
\]
It satisfies the \emph{commutation relation}
\[
\partial\circ\phi=q\phi\circ\partial
\]
with the automorphism $\phi.$

\subsection{The Weierstrass zeta function $\zeta(z,\Lambda)$ and the rings $S_{0}$
and $S$}

\subsubsection{The two rings}

From now on all the lattices $\Lambda$ will belong to the commensurability
class $\mathfrak{L}$ used to define $K$. The \emph{Weierstrass zeta
function} of $\Lambda$
\[
\zeta(z,\Lambda)=\frac{1}{z}+\sum_{0\ne\omega\in\Lambda}\left(\frac{1}{z-\omega}+\frac{1}{\omega}+\frac{z}{\omega^{2}}\right)
\]
is everywhere meromorphic on $\mathbb{C}$, has simple poles with residue $1$ at the
points of $\Lambda$ and only there, and is characterized, up to an
additive constant, by the relation $\zeta'(z,\Lambda)=-\wp(z,\Lambda).$
Its periods along $\omega\in\Lambda,$
\[
\eta(\omega,\Lambda)=\zeta(z_{0}+\omega,\Lambda)-\zeta(z_{0},\Lambda)=-\int_{z_{0}}^{z_{0}+\omega}\wp(z,\Lambda)dz
\]
(independent of $z_{0}$), are given by the Legendre $\eta$-function.
If $(\omega_{1},\omega_{2})$ is an oriented basis of $\Lambda$ (meaning
that $\mathrm{Im}(\omega_{1}/\omega_{2})>0$) and $\eta_{i}=\eta(\omega_{i},\Lambda)$,
then the \emph{Legendre period relation}
\begin{equation}\label{eq:legendre period relation}
\omega_{1}\eta_{2}-\omega_{2}\eta_{1}=2\pi i 
\end{equation}
holds. 

\begin{lem}\label{lem: real any chi}
 The space $\mathbb{C}z+\mathbb{C}\zeta(z,\Lambda)$ 
realizes every homomorphism $\chi:\Lambda\to\mathbb{C},$ \emph{ i.e.}, for
any such $\chi$ we have a unique function $h_{\chi}=az+b\zeta(z,\Lambda)$
with $a,b\in\mathbb{C}$ such that the $\omega$-period 
$$
\int_{z_{0}}^{z_{0}+\omega}dh_{\chi}=h_{\chi}(z_{0}+\omega)-h_{\chi}(z_{0})
$$
is $\chi(\omega)$ for any $\omega\in\Lambda.$
\end{lem}

\begin{proof}
The $\omega_{1}$- and $\omega_{2}$-periods of $h_{\chi}=az+b\zeta(z,\Lambda)$ are given by $a \omega_{1} + b \eta_{1}$ and $a \omega_{2} + b \eta_{2}$ respectively. 
In order to prove the lemma, we thus have to prove that, for any $(\chi_{1},\chi_{2}) \in \mathbb{C}^{2}$, there exist unique $(a,b) \in \mathbb{C}^{2}$ such that 
$$
\chi_{1} = a \omega_{1} + b \eta_{1} \text{ and } \chi_{2} = a \omega_{2} + b \eta_{2}.
$$
That this is true follows immediately from the fact that the matrix 
$$
\begin{pmatrix}
 \omega_{1} & \eta_{1}\\ 
 \omega_{2} & \eta_{2}
\end{pmatrix}
$$
is invertible because it has nonzero determinant in virtue of \eqref{eq:legendre period relation}. 
\end{proof}

\begin{lem}
(i) If $\Lambda'\subset\Lambda$ is a sublattice then
\[
\zeta(z,\Lambda)-[\Lambda:\Lambda']\zeta(z,\Lambda')\in K+ \mathbb{C}z+\mathbb{C} \subset K[z].
\]

(ii) The rings of meromorphic functions
\[
S_{0}=K[z,\zeta(z,\Lambda)]
\]
and
\[
S=K[z,z^{-1},\zeta(z,\Lambda)]
\]
do not depend on the choice of the lattice $\Lambda$ in  $\mathfrak{L}$.
\end{lem}

\begin{proof}
(i) 
The meromorphic function 
\[
\wp(z,\Lambda)-\sum_{\overline{\omega}\in\Lambda/\Lambda'}\wp(z+\omega,\Lambda')
\]
is $\Lambda$-periodic and its poles are contained
in $\Lambda$. But at 0 the poles of $\wp(z,\Lambda)$ and of $\wp(z,\Lambda')$
cancel each other, while the other terms have no pole. It follows
that this $\Lambda$-periodic function has no poles, hence is a constant.
Integrating, we find that
\[
\zeta(z,\Lambda)-\sum_{\overline{\omega}\in\Lambda/\Lambda'}\zeta(z+\omega,\Lambda')=az+b
\]
for some $a,b\in\mathbb{C}.$ On the other hand $\zeta(z+\omega,\Lambda')-\zeta(z,\Lambda')\in K_{\Lambda'}\subset K.$
It follows that
\[
\zeta(z,\Lambda)-[\Lambda:\Lambda']\zeta(z,\Lambda')\in K+ \mathbb{C}z+\mathbb{C}.
\]

(ii) This follows easily from (i). 
\end{proof}

\subsubsection{Algebraic independence of $z$ and $\zeta(z,\Lambda)$ over $K$}
\begin{lem}
\label{lem:alg indep of z and zeta}The functions $z$ and $\zeta(z,\Lambda)$
are algebraically independent over $K$.
\end{lem}

\begin{proof}
Since $K$ is algebraic over $K_{\Lambda},$ it is enough to show
that they are algebraically independent over $K_{\Lambda}$. The choice of the lattice $\Lambda$ being  clear from context, we abbreviate 
$\zeta(z,\Lambda)$ by $\zeta$. Let $(\omega_{1},\omega_{2})$
be an oriented basis of $\Lambda$. Lemma \ref{lem: real any chi} ensures that there are linear
combinations
\[
u=az+b\zeta,\,\,\,v=cz+d\zeta
\]
with coefficients $a,b,c,d \in \mathbb{C}$ such that $u$ is $\omega_{2}$-periodic
but $u(z+\omega_{1})-u(z)=1,$ while $v$ is $\omega_{1}$-periodic,
but $v(z+\omega_{2})-v(z)=1.$ Since $K_{\Lambda}(z,\zeta)=K_{\Lambda}(u,v)$,
it is enough to show that $u$ and $v$ are algebraically independent
over $K_{\Lambda}.$ Suppose there were a nontrivial algebraic relation
\[
\sum a_{ij}u^{i}v^{j}=0,
\]
with $a_{ij}\in K_{\Lambda}.$ Pick such a relation of lowest total
degree. Without loss of generality $u$ appears in it. The coefficient
of the highest power of $u$, denoted $a_{0}(v)$, is a polynomial
in $K_{\Lambda}[v]$ that does not vanish identically, as otherwise
$a_{0}(v)=0$ would be a relation of lower total degree between $u$ and $v$. Divide by it and
rewrite the relation as
\[
u^{n}+\alpha_{1}u^{n-1}+\cdots+\alpha_{n-1}u+\alpha_{n}=0
\]
where $n\ge1,$ $\alpha_{i}\in K_{\Lambda}(v).$ Let $z_{0}$ be a
point where $u(z_{0})$ and all the $\alpha_{i}(z_{0})$ are analytic,
and $z_{k}=z_{0}+k\omega_{1}.$ Using the periodicity of the $\alpha_{i}$
in $\omega_{1}$ we get, after evaluating at $z_{k}$ and using the fact that $u(z_{k})=u(z_{0})+k$, 
\[
(u(z_{0})+k)^{n}+\alpha_{1}(z_{0})(u(z_{0})+k)^{n-1}+\cdots+\alpha_{n-1}(z_{0})(u(z_{0})+k)+\alpha_{n}(z_{0})=0
\]
Thus, the polynomial $X^{n}+\alpha_{1}(z_{0})X^{n-1}+\cdots+\alpha_{n-1}(z_{0})X+\alpha_{n}(z_{0})$ has infinitely many roots and this
yields a contradiction.
\end{proof}

\subsubsection{$\partial$-simplicity of $S_{0}$ and $S$}

The derivation $\partial$ extends from $K$ to $S_{0}$ and $S$.
A ring with a derivation $\partial$ is called \emph{$\partial$-simple}
if it does not admit a non-trivial ideal $I$ invariant under $\partial.$
\begin{lem}
\label{lem:partial simplicity of S}
The rings $S_{0}$ and $S$ are $\partial$-simple.
\end{lem}

\begin{proof}
Since $S$ is a localization of $S_{0},$ it is enough to check the
assertion for $S_{0}.$ Let $\zeta=\zeta(z,\Lambda).$ The functions
$z,\zeta$ are algebraically independent over $K,$ hence every element
of $S_{0}$ has a unique expression
\[
f=\sum a_{i,j}z^{i}\zeta^{j},\,\,\,a_{i,j}\in K.
\]
Order the pairs $(i,j)$ lexicographically (first by $i,$ then by
$j$). If $f\ne0,$ we denote the maximal $(i,j)$ for which $a_{i,j}\ne0$
by $d(f).$

Let $I$ be a non-zero proper $\partial$-ideal, and consider $0\ne f\in I$
of minimal $d(f).$ Since $I$ is proper, $d(f)=(i_{0},j_{0})>(0,0).$
We may also assume that $a_{i_{0},j_{0}}=1.$ Since $\partial(z)=1$
and $\partial(\zeta)\in K$,
\[
d(\partial(f))<d(f).
\]
By our assumption, $\partial(f)\in I,$ hence by the minimality of
$d(f),$ we must have $\partial(f)=0.$ This forces $f$ to be constant,
a contradiction.
\end{proof}

The ring $S$  (but not $S_0$) is also $\phi$-simple, see Lemma \ref{lem:phi simplicity of S} below.

\subsection{A technical result on the ring $S_{0}$}

The following Proposition is a crucial ingredient in the proof of
the main theorem of the paper. However, it will be needed only at
the very end, and its proof, which is lengthy, is deferred to Appendix
A.

Besides $\partial,$ the rings $S_{0}$ and $S$ (but not $K$) carry
also the derivation
\[
\delta=z\partial=z\frac{d}{dz}.
\]
The advantage of $\delta$ over $\partial$ is that it commutes with
$\phi$: $\delta\circ\phi=\phi\circ\delta$.
\begin{prop}
\label{prop:technical_elliptic}Let $f,g\in S_{\Lambda}=K_{\Lambda}[z,\zeta(z,\Lambda)],$
$a,c\in\mathbb{C}$ and $p\in\mathbb{C}[z]$ be such that
\[
(\delta-c)(g)=(\phi-a)(f)+p.
\]
Then $g=(\phi-a)(f_{1})+p_{1}$ for some $f_{1}\in S_{\Lambda}$ and
$p_{1}\in\mathbb{C}[z].$ Furthermore, if $a=q^{r}$ for some $r\ge0$
we can take $p_{1}=dz^{r},$ $d\in\mathbb{C},$ and otherwise we can
take $p_{1}=0.$
\end{prop}
\begin{cor}
\label{cor:technical elliptic}Let $\mathscr{L}\in\mathbb{C}[\delta]$
be a monic polynomial in $\delta$, $a\in\mathbb{C}$ and $f,b\in S_{\Lambda}$ such
that
\[
\mathscr{L}(b)=(\phi-a)(f).
\]
Then $b=(\phi-a)(h)+p$ for some $h\in S_{\Lambda}$ and $p\in\mathbb{C}[z].$
Furthermore, if $a=q^{r}$ for some $r\ge0$ we can take $p=dz^{r},$
$d\in\mathbb{C},$ and otherwise we can take $p=0.$
\end{cor}

\begin{proof}
(of corollary) Factor $\mathscr{L}=(\delta-c_{k})\circ\cdots\circ(\delta-c_{1})$
with $c_{i}\in\mathbb{C},$ and use descending induction on $0\le j\le k-1$
to find $f_{j}\in S_{\Lambda}$ and $p_{j}=d_{j}z^{r}\in\mathbb{C}[z]$
(with $d_{j}=0$ if $a\notin\{1,q,q^{2},...\}$) such that
\[
(\delta-c_{j})\circ\cdots\circ(\delta-c_{1})(b)=(\phi-a)(f_{j})+p_{j}.
\]
At the end of the induction, set $h=f_{0}$ and $p=p_{0}.$
\end{proof}

\section{\label{sec:equations and modules}Linear difference and differential
equations, and their associated modules}

The purpose of this section is to review some standard results and
set up notation.

\subsection{Systems of linear difference equations and difference modules}

References for the results mentioned below may be found in  \cite[Section 1.4]{S-vdP97}.
\subsubsection{Systems and modules}
Let $(K,\phi)$ be a field of characteristic 0, equipped with an automorphism
$\phi$, called a \emph{difference operator. }

To avoid trivialities
we assume that $\phi$ is of infinite order: no positive power of
$\phi$ is the identity. We denote by
\[
C=K^{\phi}=\{ f \in K  \, | \, \phi(f)=f \}
\]
the \emph{field of} $\phi$\emph{-constants}, and assume that it is
algebraically closed, although for most of what we do, at least in
the beginning, this is not essential.

A system of linear equations
\begin{equation}
\phi(y)=Ay,\label{eq:general diference system}
\end{equation}
where $A\in \GL_{n}(K),$ is called a linear system of $\phi$-difference
equations. One seeks solutions $y\in R^{n}$ where $(R,\phi)$ is
a $\phi$-ring extension of $(K,\phi)$, \emph{i.e.}, a ring  extension together with a compatible extension of $\phi$. 

A \emph{difference module} over $K$ (called also a $\phi$\emph{-module})
is a pair $(W,\Phi)$
where $W$ is a finite dimensional $K$-vector space and $\Phi$ a
bijective $\phi$-semilinear endomorphism of $W$, {\emph i.e.}, $\Phi:W \rightarrow W$ is a bijective map such that, for all $a \in K$ and $w,w' \in W$,
$$
\Phi(aw+w')=\phi(a)\Phi(w)+\Phi(w').
$$
To the system $(\ref{eq:general diference system})$
we attach the difference module\emph{ $(K^{n},\Phi)$ }with $\Phi(y)=A^{-1}\phi(y).$
Conversely, if $(W,\Phi)$ is a difference module, $e_{1},...,e_{n}$
is a basis of $W$ and
\[
\Phi(e_{j})=\sum_{i=1}^{n}a_{ij}e_{i},
\]
we attach to $(W,\Phi)$ the system $(\ref{eq:general diference system})$
with $A^{-1}=(a_{ij}).$ This process is not canonical, as it depends on the chosen basis $e_{1},...,e_{n}$ of $W$. If $\widetilde{e}_{1},...,\widetilde{e}_{n}$ is another basis of
the same module and
\[
e_{j}=\sum_{i=1}^{n}p_{ij}\widetilde{e}_{i},
\]
so that $P=(p_{ij})$ is the change-of-basis matrix, then we get the system $\phi(y)=\widetilde{A}y$ where 
\[
\widetilde{A}=\phi(P)AP^{-1}.
\]
Two matrices $A$ and $\widetilde{A}$ related by such a relation
with $P\in \GL_{n}(K)$ are said to be \emph{gauge equivalent} over $K$. The
above procedure establishes a bijection between systems $(\ref{eq:general diference system}),$
up to gauge equivalence over $K$, and difference modules $(W,\Phi),$ up to
isomorphism over $K$.

A difference module isomorphic to $(K^{n},\phi)$ is called \emph{trivial}.
The module associated to $(\ref{eq:general diference system})$ is
trivial if and only if $A$ is gauge equivalent to the identity matrix.

\subsubsection{Base change and solutions set}

If $(R,\phi)$ is a $\phi$-ring extension of $(K,\phi)$ and $(W,\Phi)$
is a difference module over $K$, then we may consider its base change
\[
(W_{R},\Phi_{R})=(R\otimes_{K}W,\phi\otimes\Phi).
\]
Note that $W_{R}$ is free of rank $n=\dim_{K}W$ over $R.$ Let $C_{R}=R^{\phi}$
be the ring of $\phi$-constants of $R$. If $R$ is a \emph{ $\phi$-simple}
$\phi$-ring (does not have any non-trivial $\phi$-ideal),
and in particular if it is a field, then $C_{R}$ is a field. Assume
from now on that $R$ is $\phi$-simple. We continue to denote $\Phi_{R}$
simply by $\Phi$.
The set  $W_{R}^{\Phi}$ of $\Phi$-fixed vectors   of $W_{R}$ is a vector
space over $C_{R}.$ The Casoratian Lemma asserts that
\begin{equation}
R\otimes_{C_{R}}W_{R}^{\Phi}\hookrightarrow R\otimes_{K}W\label{eq:Wronskian lemma embedding}
\end{equation}
is injective, and in particular $\dim_{C_{R}}W_{R}^{\Phi}\le\dim_{K}W,$
as can be seen by calculating the $R$-ranks of the two sides.

If $(W,\Phi)=(K^{n},A^{-1}\phi)$ is the difference module attached
to the system (\ref{eq:general diference system}), then the fixed
vectors $\mathcal{U}_{R}=W_{R}^{\Phi}$ are the solutions of the system
in $R.$ For this reason we sometimes call it the \emph{solution set}.
We say that the system $(\ref{eq:general diference system})$ attains
a \emph{full set of solutions} over a $\phi$-simple $\phi$-ring $R$ if
$\dim_{C_{R}}\mathcal{U}_{R}=n$. This is equivalent to the existence
of a matrix $U\in \GL_{n}(R)$ satisfying $\phi(U)=AU.$ Such a matrix
is called a \emph{fundamental matrix} over $R$, and its columns span
$\mathcal{U}_{R}=UC_{R}^{n}.$ If $U'$ is another fundamental matrix
over $R$ then
\[
U'=UT
\]
with $T\in \GL_{n}(C_{R}).$ Yet another way to say that $(\ref{eq:general diference system})$
attains a full set of solutions over $R$ is that $(W_{R},\Phi_{R})$
is trivial: the embedding $(\ref{eq:Wronskian lemma embedding})$
is an isomorphism.

Assume that $W$ attains a full set of solutions over the  $\phi$-simple
$\phi$-ring $R$ and $R'$ is a $\phi$-ring extension of $R$ (not
necessarily $\phi$-simple) such that $C_{R'}=C_{R}.$ If $U\in \GL_{n}(R)$
is a fundamental matrix and $v\in W_{R'}^{\Phi}$, then $U^{-1}v$
is fixed by $\phi,$ hence $v\in UC_{R}^{n}=\mathcal{U}_{R}.$ Thus
the space of solutions does not grow when $R$ grows, as long as the
ring of $\phi$-constants stays the same.

\subsubsection{An irreducibility criterion}

We shall need the following lemma.  Compare Lemma 4.6
of \cite{A-D-H-W}, where it is deduced from the cyclic vector lemma.
We give a more direct proof here. A $\phi$-module $W$ over $K$
is called irreducible if it does not admit any $\phi$-submodule other
than $0$ and $W$ itself.
\begin{lem}
\label{lem:Irreducibility lemma}Let $K'=K(z)$ be a transcendental
extension of $K$ and $W$ a $\phi$-module over $K$. Extend $\phi$
to $K'$ by the rule $\phi(z)=qz,$ where $q\in C^{\times},$ $C=K^{\phi}.$
Then $W$ is irreducible if and only if $W_{K'}$ is irreducible.
\end{lem}

\begin{proof}
It is clear that if $W$ is reducible, so is $W_{K'}.$ To prove the
converse, let $k=K((z))$ and $R=K[[z]],$ and extend the action of
$\phi$ to them, so that $K\subset K'\subset k$ is a tower of $\phi$-fields.
If $W_{K'}$ is reducible, so is $W_{k},$ so it is enough to prove
that if $W_{k}$ is reducible, so is $W$.

Let $V\subset W_{k}$ be a $\phi$-submodule over $k$, with $0<\dim_{k}V<\dim_{K}W.$
Then $V_{0}=V\cap W_{R}$ is a $\phi$-submodule of $W_{R}$ over
$R$. Since $W_{R}$ is a free $R$-module and since $R$ is a principal ideal domain, $V_{0}$ is a free $R$-module as well. Using the fact that $k$ is the quotient ring of $R$, it is easily seen that the maximal number of $R$-linearly independent elements of $V_{0}$ is equal to the maximal number of $k$-linearly independent elements of $V$, so $\mathrm{rk}_{R}V_{0}=\dim_{k}V$.

The inclusion  $V_{0} \subset W_{R}$ of $\phi$-modules over $R$ induces the monomorphism 
$$
\overline{V_{0}}:=V_{0}/(V_{0} \cap z W_{R}) \hookrightarrow \overline{W_{R}}:=W_{R}/zW_{R}
$$ 
of $\phi$-modules over $K$. Note that $V_{0} \cap z W_{R}=V \cap zW_{R}=z(V \cap W_{R})=z V_{0}$, so $\overline{V_{0}}=V_{0}/zV_{0}$ and, hence, $\dim_{K} \overline{V_{0}}= \mathrm{rk}_{R}V_{0}=\dim_{k}V$. Moreover, $\overline{W_{R}}$ and  $W$ are clearly isomorphic as $\phi$-modules over $K$. Therefore, $W$ has a $\phi$-submodule over $K$ of dimension $\dim_{k}V$ and, hence, is reducible.

\end{proof}

\begin{rem}
With the notations of the previous proof, we stress that, in general, the base change of $\overline{V_{0}}$
to $k$ need not be $V$, and the submodule $V$ need not descend
to $K$.
\end{rem}

\subsection{Systems of linear differential equations and modules with connections}

This analogue of the difference equation set-up is even more classical,
see \cite{S-vdP03}. We give it only to fix notation.

\subsubsection{Systems and modules}

Let $(K,\partial)$ be a field of characteristic 0, equipped with
a non-trivial derivation. We denote by $C=K^{\partial}$ the kernel
of $\partial,$ and call it the field of $\partial$\emph{-constants}.
Although not essential for many of the arguments, we assume that $C$
is algebraically closed.

Assume we are given a \emph{linear system of differential equations}
\begin{equation}
\partial(y)=By,\label{eq:general differential system}
\end{equation}
where $B\in\mathfrak{gl}_{n}(K)$. One seeks solutions in $\partial$-ring extensions
$(R,\partial)$ of $K$.

A \emph{module with a connection} $(W,\nabla)$ over $K$ (called
also a $\partial$-module) is a finite dimensional vector space over
$K$ equipped with a $\partial$-connection, \emph{i.e.}, an additive map $\nabla :M \rightarrow M$ such that $\nabla( a w)=\partial(a)w + a \nabla(w)$ for all $a$ in $K$ and $w$ in $W$. To the system $(\ref{eq:general differential system})$
one attaches the module
\[
W=K^{n},\,\,\,\nabla(y)=\partial(y)-By.
\]
Conversely, if $(W,\nabla)$ is a module with a connection and $e_{1},...,e_{n}$
is a basis of $W$, then
\[
\nabla(e_{j})=\sum_{i=1}^{n}b_{ij}e_{i},
\]
and we associate to it the system $(\ref{eq:general differential system})$
with $B=-(b_{ij}).$ If $\widetilde{e}_{1},...,\widetilde{e}_{n}$
is another basis of the same module and
\[
e_{j}=\sum_{i=1}^{n}p_{ij}\widetilde{e}_{i},
\]
then the matrix giving $\nabla$ in the new basis is $-\widetilde{B}$
where
\[
\widetilde{B}=PBP^{-1}+\partial P\cdot P^{-1}.
\]
We call a pair of matrices $B,\widetilde{B}$ related as above \emph{gauge
equivalent}. The above procedure establishes a bijection between systems
$(\ref{eq:general differential system}),$ up to gauge equivalence,
and modules with a connection $(W,\nabla),$ up to isomorphism.

A \emph{trivial }$\partial$-module is a module isomorphic to $(K^{n},\partial)$.
The module associated to $(\ref{eq:general differential system})$
is trivial if and only if $B$ is gauge equivalent to the zero matrix.

\subsubsection{Base change and solutions set}

If $(R,\partial)$ is a $\partial$-ring extension of $(K,\partial)$
and $(W,\nabla)$ is a $\partial$-module over $K$, then we may consider
its base change
\[
(W_{R},\nabla_{R})=(R\otimes_{K}W,\partial\otimes1+1\otimes\nabla).
\]
Note that $W_{R}$ is free of rank $n=\dim_{K}W$ over $R.$ Let $C_{R}=R^{\partial}$
be the ring of $\partial$-constants of $R$. If $R$ is a  \emph{$\partial$-simple}
$\partial$-ring (does not have any non-trivial $\partial$-ideal), and in particular if it is a field, then $C_{R}$ is a field.
Assume from now on that $R$ is $\partial$-simple, and recall that
$\partial$-simple rings are domains (have no zero divisors). 

The  kernel of $\nabla$ in $W_{R}$, denoted $W_{R}^{\nabla},$ is
a vector space over $C_{R}.$ The Wronskian Lemma asserts that
\begin{equation}
R\otimes_{C_{R}}W_{R}^{\nabla}\hookrightarrow R\otimes_{K}W\label{eq:Wronskian lemma embedding-1}
\end{equation}
is injective, and in particular $\dim_{C_{R}}W_{R}^{\nabla}\le\dim_{K}W,$
as can be seen by calculating the $R$-ranks of the two sides.

If $(W,\nabla)=(K^{n},\partial-B)$ is the $\partial$-module attached
to the system $(\ref{eq:general differential system})$, then $\mathcal{U}_{R}=W_{R}^{\nabla}$
are the solutions of the system over $R.$ For this reason we sometimes
call it the \emph{solution set}. We say that the system $(\ref{eq:general differential system})$
attains a \emph{full set of solutions} over a $\partial$-simple $\partial$-ring
$R$ if $\dim_{C_{R}}\mathcal{U}_{R}=n$. This is equivalent to the
existence of a matrix $U\in \GL_{n}(R)$ satisfying $\partial(U)=BU.$
Such a matrix is called a \emph{fundamental matrix} over $R$, and
its columns span $\mathcal{U}_{R}=UC_{R}^{n}.$ If $U'$ is another
fundamental matrix over $R$ then
\[
U'=UT
\]
with $T\in \GL_{n}(C_{R}).$ Yet another way to say that $(\ref{eq:general differential system})$
attains a full set of solutions over $R$ is that $(W_{R},\nabla_{R})$
is trivial: the embedding $(\ref{eq:Wronskian lemma embedding-1})$
is an isomorphism.

Assume that $W$ attains a full set of solutions over the $\partial$-simple
$\partial$-ring $R$ and $R'$ is a $\partial$-ring extension of
$R$ (not necessarily $\partial$-simple) such that $C_{R'}=C_{R}.$
If $U\in \GL_{n}(R)$ is a fundamental matrix and $v\in W_{R'}^{\nabla}$,
then $U^{-1}v$ is killed by $\partial,$ hence $v\in UC_{R}^{n}=\mathcal{U}_{R}.$
Thus the space of solutions does not grow when $R$ grows, as long
as the ring of $\partial$-constants stays the same.

\subsection{Systems of linear $(\phi,\partial)$-equations and $(\phi,\partial)$-modules}

\subsubsection{$(\phi,\partial)$-fields}

Consider a triplet $(K,\phi,\partial)$ where $K$ is a field of characteristic
zero, $\phi$ is an automorphism, and $\partial$ a derivation satisfying
\[
\partial\circ\phi=q\phi\circ\partial.
\]

We assume that $q\in K^{\phi}\cap K^{\partial}.$ Such a triplet will
be called a $(\phi,\partial)$-field. The main examples considered
in this work are:
\begin{itemize}
\item $K,\phi$ and $\partial$ are as in \S~\ref{subsec:K phi and partial}.
\item The same example, where $K$ is replaced by the field $M$ of meromorphic
functions on $\mathbb{C}$, or by the field
$F=\mathbb{C}((z)).$
\item $K$ is replaced by $K'=K(z)\subset M$ or by $F$, $\phi$ is the
same, but $\partial$ is replaced by the derivation $\delta=z\partial.$
In this example $\delta\circ\phi=\phi\circ\delta,$ so the $q$ appearing
in the commutation relation is not the $q$ defining $\phi,$ but
1.
\end{itemize}
We let $C$ be the field of $\phi$-constants in $K$.
In the examples mentioned above $C=\mathbb{C}$.  It coincides
with the field $K^{\partial} =\{f  \, | \, \partial(f)=0\}$ of $\partial$-constants, so we simply
call it the field of constants. However, we shall have to consider
later on an extension $(\widetilde{K}',\phi,\delta)$ of $(K(z), \phi,\delta)$, in which the $\delta$-field$(\widetilde{K}')^{\phi}=\widetilde{C}$
is a \emph{differential closure} of $\mathbb{C}$. The differential closure $\widetilde{C}$  of $\mathbb{C}$ is
a $\delta$-field extension of $\mathbb{C}$, unique up to $\delta$-isomorphism,
characterized by the following two properties:
\begin{itemize}
\item  $\widetilde{C}$ is \emph{differentially closed}: Every finite system of polynomial differential equations in several variables in the operator
$\delta$, with coefficients
from $\widetilde{C}$, which has
a solution in some $\delta$-field extension of $\widetilde{C}$,
already has a solution in $\widetilde{C}.$
\item If $(\mathscr{C},\delta)$ is another differentially closed field
containing $\mathbb{C}$ then there exists a $\delta$-embedding of
$\widetilde{C}$ in $\mathscr{C}\text{ over \ensuremath{\mathbb{C}}}.$
\end{itemize}
See \cite{M-MTDF} for a survey, including proofs of the existence
and uniqueness of $\widetilde{C}.$ The existence is due to Kolchin
and (with a simpler characterization than the one given here) to Blum.
The uniqueness is due to Shelah. Blum and Shelah's approaches are proved by model theoretic
means; differentially closed fields do not show up ``naturally''
in algebra, and are considered ``huge''. Despite the analogy with
the algebraic closure of a field, caution must be exercised. For example,
$\widetilde{C}$ fails to satisfy \emph{minimality}: it has proper
subfields $\delta$-isomorphic to it over $\mathbb{C}.$

It is proven in Proposition 2.11 of loc.cit. that the $\delta$-constants
of a $\delta$-closure $\widetilde{C}$ of $\mathbb{C}$ are just
$\mathbb{C}$.  Proposition 1.1 of \cite{M-RDCF} proves that the group
$\mathrm{Aut}_{\delta}(\widetilde{C}/\mathbb{C})$ of $\delta$-automorphisms
of $\widetilde{C}/\mathbb{C}$ is rich: the only elements of $\widetilde{C}$
fixed by $\mathrm{Aut}_{\delta}(\widetilde{C}/\mathbb{C})$
are the elements of $\mathbb{C}.$ Both these facts will play a role
later on.

\subsubsection{Systems and modules }

Given a $(\phi,\delta)$-field $K$, we consider the double system of equations
\[
\begin{cases}
\begin{array}{c}
\phi(y)=Ay\\
\partial(y)=By
\end{array}\end{cases}
\]
where $A\in \GL_{n}(K)$ and $B\in\mathfrak{gl}_{n}(K).$ One seeks
solutions in $(\phi,\partial)$-ring extensions $R$ of $K$. It is
readily checked that a necessary condition for the existence of a
$U \in \GL_{n}(R),$ for some extension $R,$ satisfying both sets of
equations, is the \emph{consistency relation}
\begin{equation}
q\phi(B)=ABA^{-1}+\partial A\cdot A^{-1}.\label{eq:consistencey condition}
\end{equation}
If this relation holds we call the above system of equations \emph{consistent}.

A $(\phi,\partial)$\emph{-module} over $K$ is a triple $(W,\Phi,\nabla)$
such that $(W,\Phi)$ is a $\phi$-module, $(W,\nabla)$ is a $\partial$-module
and
\[
\nabla\circ\Phi=q\Phi\circ\nabla.
\]
The $(\phi,\partial)$-module associated to a consistent system as
above is $(K^{n},\Phi,\nabla)$ where $\Phi(y)=A^{-1}\phi(y)$ and
$\nabla(y)=\partial(y)-By.$ The consistency relation $(\ref{eq:consistencey condition})$
between $A$ and $B$ is equivalent to the commutation relation between
$\Phi$ and $\nabla.$ Conversely, if $(W,\Phi,\nabla)$ is a $(\phi,\partial)$-module,
$e_{1},...,e_{n}$ is a basis of $W$ over $K$ and $A$ and $B$
are constructed as in the previous sections, then we get a consistent
system of equations.

If we change the basis as before, the new pair of matrices that we
obtain, with respect to the new basis, is
\[
(\widetilde{A},\widetilde{B})=(\phi(P)AP^{-1},PBP^{-1}+\partial P\cdot P^{-1}).
\]
We call two pairs $(\widetilde{A},\widetilde{B})$ and $(A,B)$ related
in this way \emph{gauge equivalent}. As before, this establishes a
bijection between consistent systems as above, up to gauge equivalence,
and isomorphism classes of $(\phi,\partial)$-modules.

The module $(W,\Phi,\nabla)$ is called \emph{trivial} if it is isomorphic
to the module $(K^{n},\phi,\partial).$ In terms of the associated
systems, this means that $(A,B)$ is gauge equivalent to $(I,0).$
It may happen that $(W,\Phi)$ or $(W,\nabla)$ are trivial but $(W,\Phi,\nabla)$
is not. In such a case we use the terminology \emph{$\phi$-trivial
}or \emph{$\partial$-trivial}.

A particular case of gauge equivalence occurs when we take the transition
matrix between the bases to be $P=A.$ In this case we get, thanks
to the consistency relation, the pair
\[
(\widetilde{A},\widetilde{B})=(\phi(A),q\phi(B)).
\]

\subsubsection{Base change and solution sets}

It is tempting to look for a $(\phi,\partial)$-ring extension $R$
over which the \emph{$(\phi,\partial)$-}system above attains a full
set of solutions. While the definitions can be immitated, at this
point we take a different path. We shall be primarily interested in
the \emph{difference }Picard Vessiot theory,\emph{ i.e.}, the construction
of a ``minimal'' ring over which $(W,\Phi)$ alone is trivialized.
We shall show that \emph{if} $W$ carries a consistent $\partial$-structure,
\emph{i.e.}, can be extended to become a $(\phi,\partial)$-module, this imposes
serious restrictions on the \emph{difference Galois group} of the
system $\phi(y)=Ay.$ Similarly, we might be interested in a $\partial$-module
$(W,\nabla),$ its \emph{differential }Picard Vessiot theory, and
the restrictions that a consistent $\phi$-structure imposes on its
\emph{differential Galois group}. This leads us to the notions of
$\partial$-integrability and $\phi$-isomonodromy\footnote{Some authors use $\phi$-integrability, to stress the analogy with
$\partial$-integrability. However, as there are no differential equations
involved with $\phi,$ we prefer the term $\phi$-isomonodromy.} discussed below.

\subsection{$\partial$-integrable $\phi$-modules}

Let $K$ be a $(\phi,\partial)$-field. Let $(W,\Phi)$ be a $\phi$-module
over $K$ and $\phi(y)=Ay$ the associated system of difference equations,
in some basis of $W$. Concretely, we may take $(W,\Phi)=(K^{n},A^{-1}\phi).$ 

Let $K[\varepsilon]$ be the ring of dual numbers over $K$ ($\varepsilon^{2}=0$).
Extend the automorphism $\phi$ to $K[\varepsilon]$ as $\phi(a+b\varepsilon)=\phi(a)+q\phi(b)\varepsilon,$
and define a ring homomorphism $h_{\partial}:K\to K[\varepsilon]$
by $h_{\partial}(a)=a+\partial(a)\varepsilon.$ Both the inclusion
$K\hookrightarrow K[\varepsilon]$ and $h_{\partial}$ commute with
$\phi$. Consider the two $\phi$-modules
\[
W_{\varepsilon}=K[\varepsilon]\otimes_{K}W,\,\,\,\Phi_{\varepsilon}=\phi\otimes\Phi
\]
and
\[
W_{\varepsilon}^{(\partial)}=K[\varepsilon]\otimes_{h_{\partial},K}W,\,\,\,\Phi_{\varepsilon}^{(\partial)}=\phi\otimes\Phi.
\]

One can thereby construct an endofunctor of the category of $\phi$-modules over $K$ which attaches to any module $W$ the module $W_{\varepsilon}^{(\partial)}$. Such a functor is called the \emph{prolongation functor} and  has been introduced by Ovchinnikov~\cite{Ovdifftan} and Kamensky~\cite{Kamdifftan} when defining  differential Tannakian categories. The $\phi$-module  $W_{\varepsilon}^{(\partial)}$ is an extension of $W$ by itself in the category of $\phi$-modules over $K$. Moreover, if $\phi(y)=Ay$ is a system associated to $W$ then $\phi(y)=\begin{pmatrix}
A & \partial(A) \\
0 &A
\end{pmatrix} y$ is a system associated to $W_{\varepsilon}^{(\partial)}$.

\begin{defn}
The $\phi$-module $W$ is called $\partial$\emph{-integrable} if
there exists an isomorphism of $\phi$-modules $(W_{\varepsilon},\Phi_{\varepsilon})\simeq(W_{\varepsilon}^{(\partial)},\Phi_{\varepsilon}^{(\partial)})$
over $K[\varepsilon],$ reducing to the identity modulo $\varepsilon$.
\end{defn}
It is easily seen that a module $W$ is $\partial$\emph{-integrable} if and only if the extension of $W$ by itself corresponding to $W_{\varepsilon}^{(\partial)}$ splits. Moreover, we can interpret the notion of $\partial$-integrability in the following terms.
\begin{lem}\label{lem:partialintmoduldestructure}
Let $(W,\Phi)$ be a $\phi$-module over $K.$ Then $W$ is $\partial$-integrable
if and only if it carries a structure of a $(\phi,\partial)$-module
extending the given $\phi$-module structure. Explicitly, this means
that there exists a $\partial$-connection
\[
\nabla:W\to W
\]
such that $q\Phi\circ\nabla=\nabla\circ\Phi$. In terms of the matrix
$A,$ the module $W$ (or system $(\ref{eq:general diference system})$)
is $\partial$-integrable if and only if there exists a matrix $B\in\mathfrak{gl}_{n}(K)$
such that
\[
\partial(A)=q\phi(B)A-AB.
\]
\end{lem}

\begin{proof}
Let $\iota:(W_{\varepsilon}^{(\partial)},\Phi_{\varepsilon}^{(\partial)})\simeq(W_{\varepsilon},\Phi_{\varepsilon})$
be an isomorphism of $\phi$-modules over $K[\varepsilon],$ reducing to the identity modulo $\varepsilon$. Write $W_{\varepsilon}^{(\partial)}=1\otimes W+\varepsilon\otimes W.$
Note that only the second summand is a $K$-subspace, as 
\[
\lambda\otimes w=1\otimes\lambda w-\partial(\lambda)\varepsilon\otimes w\,\,\,(\lambda\in K,\,\,w\in W)
\]
need not lie in $1\otimes W.$ Nevertheless, as $\iota$ is $K[\varepsilon]$-linear,
it is determined by the values $\iota(1\otimes w).$ Since $\iota$
reduces to the identity modulo $\varepsilon$ we can define $\nabla$
by
\[
\iota(1\otimes w)-1\otimes'w=\varepsilon\otimes'\nabla(w).
\]
To avoid confusion we used $\otimes'$ for the tensor product in $W_{\varepsilon}.$
We then have 
\[
\varepsilon\otimes'\nabla(\lambda w)=\iota(1\otimes\lambda w)-1\otimes'\lambda w=\iota(\lambda\otimes w+\partial(\lambda)\varepsilon\otimes w)-\lambda\otimes'w.
\]
This equals $\lambda\varepsilon\otimes'\nabla(w)+\partial(\lambda)\varepsilon\otimes'w.$
It follows that
\[
\nabla(\lambda w)=\lambda\nabla(w)+\partial(\lambda)w,
\]
so $\nabla$ is indeed a $\partial$-connection. That it satisfies
$q\Phi\circ\nabla=\nabla\circ\Phi$ follows from $\iota\circ\Phi_{\varepsilon}^{(\partial)}=\Phi_{\varepsilon}\circ\iota.$
These arguments can be reversed, proving the Lemma.
\end{proof}
Note that any difference system $\phi(y)=Ay$ with $A \in \GL_n(K^\partial)$ is $\partial$-integrable (take $B=0$).
\subsection{$\phi$-isomonodromic $\partial$-modules}

Still assuming that $K$ is a $(\phi,\partial)$-field, let $(W,\nabla)$
be a module with a connection, and $\partial(y)=By$ the associated
system of differential equations, in some basis of $W$. Concretely,
we may take $(W,\nabla)=(K^{n},\partial-B).$

The module $(W^{(\phi)},\nabla^{(\phi)})$ is defined as
\[
W^{(\phi)}=K\otimes_{\phi,K}W,\,\,\,\nabla^{(\phi)}(a\otimes w)=\partial a\otimes w+qa\otimes\nabla w.
\]
It is again a $\partial$-module. 
\begin{defn}
A $\partial$-module $W$ is called \emph{$\phi$-isomonodromic }(or
$\phi$\emph{-integrable}) if $W\simeq W^{(\phi)}.$
\end{defn}

\begin{lem}
\label{lem:TFAE for =00005Cphi-integrability-1}Let $(W,\nabla)$
be a $\partial$-module over $K$. Then $W$ is $\phi$-isomonodromic
if and only if it carries a structure of a $(\phi,\partial)$-module
extending the given $\partial$-module structure. In terms of the
matrix $B$ the module $W$ (or the system $(\ref{eq:general differential system})$)
is $\phi$-isomonodromic if and only if there exists an $A\in \GL_{n}(K)$
such that
\[
q\phi(B)=ABA^{-1}+\partial A\cdot A^{-1},
\]
\emph{i.e.}, if and only if $B$ and $q\phi(B)$ are gauge equivalent.
\end{lem}

\begin{proof}
If $\iota:W^{(\phi)}\simeq W$ is an isomorphism of $\partial$-modules,
put $\Phi(w)=\iota(1\otimes w).$ Then $\Phi(aw)=\phi(a)\cdot\Phi(w)$
and the relation $\nabla\circ\Phi=q\Phi\circ\nabla$ follows from
$\iota\circ\nabla^{(\phi)}=\nabla\circ\iota$. Conversely, given $\Phi$
define $\iota(a\otimes w)=a\Phi(w).$
\end{proof}
\begin{example}
\label{exa:basic example-1}Let $K$ be as in \S~\ref{subsec:K phi and partial}.
If $f\in F=\mathbb{C}((z))$ satisfies a linear differential equation of degree $n$
with coefficients from $K,$ as well as a $\phi$-difference equation
of degree $m$ over $K,$ then $W=Span_{K}\{\phi^{i}\partial^{j}f|\,0\le i<m,\,\,0\le j<n\}\subset F$
is a $(\phi,\partial)$-module over $K$, when we let $\nabla=\partial$
and $\Phi=\phi$.

If $f$ only satisfied a differential equation and we take $N=Span_{K}\{\partial^{j}(f)\}$
then $(N,\nabla)=(N,\partial)$ is a $\partial$-module, and $\alpha:1\otimes g\mapsto\phi(g)$
yields $\alpha:(N^{(\phi)},\nabla^{(\phi)})\simeq(\phi(N),\partial)$
(as a subspace of $F$). Indeed, for $g\in N$
\[
\alpha\circ\nabla^{(\phi)}(1\otimes g)=\alpha(q\otimes\partial g)=q\phi(\partial(g))=\partial(\phi(g))=\partial\circ\alpha(1\otimes g).
\]

If $N=\phi(N)$ then clearly $N\simeq\phi(N)$ and $N=W$ has the
structure of a $(\phi,\partial)$-module.
\end{example}

\section{\label{sec:Picard-Vessiot-and-parametrized}Picard-Vessiot and parametrized
Picard-Vessiot theory}

\subsection{Picard Vessiot theory of difference equations}
A classical exposition of Picard-Vessiot theory for difference equations  is \cite{S-vdP97} whereas parametrized Picard-Vessiot theory is introduced in \cite{H-S08}. The articles \cite{O-W}, \cite{dV-H} and \cite{Wi12} contain some schematic approaches.
\subsubsection{The Picard Vessiot extension and the Picard Vessiot ring}\label{subsec:PVextPVring}

Let $(K,\phi)$ be a $\phi$-field of characteristic 0, with an algebraically
closed field of constants $C=K^{\phi}$. Consider the system of equations
$(\ref{eq:general diference system}).$ One is interested in constructing
a minimal $\phi$-extension $(L,\phi)$ in which it attains a full
set of solutions, much like the construction of a splitting field
of a polynomial in Galois theory. Easy examples show that $L$ can
not always be a field. The next best option works.
\begin{defn}
(i) A $\phi$\emph{-pseudofield} $L$ is a direct product of finitely
many fields
\[
L=L_{1}\times\cdots\times L_{r},
\]
together with an automorphism $\phi,$ permuting the $L_{i}$ cyclically:
$\phi(L_{i})=L_{i+1\mod r}.$

(ii) A \emph{Picard Vessiot (PV) extension} of $K$ for $(\ref{eq:general diference system})$
is a $\phi$-pseudofield $(L,\phi)$ containing $(K,\phi)$, satisfying:
\begin{itemize}
\item There exists a fundamental matrix $U\in \GL_{n}(L)$ for $(\ref{eq:general diference system}),$
and $L=K(U),$ \emph{i.e.}, $L$ is generated, as a pseudofield, by the entries of $U$.
\item $L^{\phi}=C$.
\end{itemize}
\end{defn}

Picard Vessiot extensions exist, and are unique up to $K$-$\phi$-isomorphism \cite[Prop.2.14 and Theorem 2.16]{O-W}.
 If $L$ is a PV extension  for \eqref{eq:general diference system} as above then  $L_{1}$ is a  PV extension
of the system
\[
\phi^{r}(y)=A^{(r)}y,\,\,\,A^{(r)}=\phi^{r-1}(A)\cdots\phi(A)A,
\]
 over the difference field  $(K, \phi^r)$. The projection of $U$ to $L_{1}$ is a fundamental matrix for
this new system. In many cases, depending on the desired conclusion,
it is harmless to replace the original system by this new one, $\phi$ by $\phi^r$, reducing
the situation to the case where $L$ is a field.

Since the fundamental matrix is unique up to multiplication on the
right by a matrix from $\GL_{n}(C),$ the subring
\[
R=K[U,\det(U)^{-1}]
\]
of $L$ is canonical, and is invariant under $\phi$ as well. It turns
out to be $\phi$-simple. A $\phi$-simple $K$-algebra generated
by a fundamental matrix $U$ and the inverse of its determinant (thus
making $U\in \GL_{n}(R)$) is called a PV \emph{ring }for $(\ref{eq:general diference system}).$
It is again unique up to isomorphism, and its total ring of fractions
(its localization in the set of non zero divisors) is a PV extension  \cite[Proposition 2.14]{O-W}.

If $(L,\phi)$ is a PV extension of $(K,\phi)$ for the system $(\ref{eq:general diference system})$
and $(K',\phi)$ is an intermediate $\phi$-extension, $K\subset K'\subset L,$
then $(L,\phi)$ is a PV extension of the same system over $K'$ as
well.
\begin{rem}
\label{rem:PV  ext containing a given solution}Let $(F,\phi)$ be
a $\phi$-field extension of $(K,\phi)$ containing the coordinates
of a solution $0\ne u\in F^{n}$ of $\phi(u)=Au$. Assume $F^{\phi}=C.$
Then we may assume, without loss of generality, that the PV extension
$L$ contains the subfield $K(u)$ of $F$. Indeed, replace $(K,\phi)$
by $(K(u),\phi)$ and look for a PV extension $L$ for the original
system over $K(u).$ Then $L$ will satisfy the requirements to become
a PV extension of the same system also over $K$.
\end{rem}

\subsubsection{The ring $S$ as a  Picard-Vessiot ring}

In the notation of \S~\ref{sec:Elliptic-functions}, the automorphism $\phi$ also extends to $S_{0}$ and $S$. The following holds.

\begin{lem}
\label{lem:phi simplicity of S}
The ring $S$ is a Picard-Vessiot ring over $K$ and thereby is $\phi$-simple.
\end{lem}

\begin{proof}
One can give a proof along the lines of Lemma \ref{lem:partial simplicity of S}. However, we
give a slicker proof, that will be useful later. Consider the system
of $\phi$-difference equations
\[
\phi(y)=\left(\begin{array}{cc}
q & f_{\zeta}(z,\Lambda)\\
0 & 1
\end{array}\right)y
\]
where $f_{\zeta}(z,\Lambda)=\zeta(qz,\Lambda)-q\zeta(z,\Lambda)\in K_{\Lambda}.$
The matrix
\[
U=\left(\begin{array}{cc}
z & \zeta(z,\Lambda)\\
0 & 1
\end{array}\right)
\]
is a fundamental matrix, and its determinant is $z.$ The field $E=K(z,\zeta(z,\Lambda))$
is a Picard-Vessiot extension for the system over $K$ because it
is generated by the entries of $U$, and its field of $\phi$-constants
is $\mathbb{C}.$ The last point is most easily seen if we embed $E$
in $F=\mathbb{C}((z))$ and use the fact that $F^{\phi}=\mathbb{C}.$
The subring $S\subset E$ is generated by the entries of $U$ and
$\det(U)^{-1},$ hence is the Picard-Vessiot \emph{ring} for the system.
But Picard-Vessiot rings are known to be $\phi$-simple \cite[Prop. 2.14]{O-W}  by general
principles. 
\end{proof}
The ring $S_{0}$ is not $\phi$-simple. The proper ideal $(z)$ is
$\phi$-invariant. The $\phi$-simplicity of $S$ has the following
useful consequence.
\begin{cor}
\label{cor: Crterion for elts of E to be in S}Let $E$ be the field
of fractions of $S$. The following properties relative to an $f\in E$ are equivalent:
\begin{itemize}
 \item $f \in S$;
 \item the $K$-vector space
 \[
M_{f}=\mathrm{Span}_{K}\{\phi^{i}(f)|\,i\ge0\}\subset E
\]
is finite dimensional.
\end{itemize}
\end{cor}

\begin{proof}
Assume that $\dim_{K}M_{f}<\infty$. We first claim that $\phi(M_{f})=M_{f}$. Indeed, since $\phi(M_{f})=\mathrm{Span}_{K}\{\phi^{i}(f)|\,i\ge1\}$, it is sufficient to prove that $f$ belongs to $\phi(M_{f})$. Let
\[
\sum_{i=0}^{m}a_{i}\phi^{i}(f)=0
\]
($a_{i}\in K)$ be a nontrivial linear relation between the $\phi^{i}(f)$ with
smallest possible $m.$ Then $a_{0}\ne0,$ for otherwise applying
$\phi^{-1}$ to the relation would give a similar one with $m-1$
replacing $m$. It follows that $a_{0}f$, and hence $f$, belongs
to $\mathrm{Span}_{K}\{\phi^{i}(f)|\,i\ge1\}=\phi(M_{f})$. This concludes the proof of our claim. 

Let
\[
I=\{a\in S|\,aM_{f}\subset S\}.
\]
Since $M_{f}$ is spanned by $\phi^{i}(f)$ for $0\le i\le m-1,$
a common denominator of these $m$ generators belongs to $I,$ so
$I\ne0.$ The set $I$ is clearly an ideal in $S,$ and it is $\phi$-invariant
because if $a\in I$ then $\phi(a)M_{f}=\phi(a)\phi(M_{f})=\phi(aM_{f})\subset S,$ so
$\phi(a)\in I.$ By the $\phi$-simplicity of $S$ we must have $1\in I,$
and in particular $f\in S.$

For the converse, note that the set of $f$'s for which $\dim_{K}{M_f}$ is finite is a $K$-algebra. As it contains $z,z^{-1}$ and $\zeta(z,\Lambda)$, it contains the ring $S$.
\end{proof}

\subsubsection{The difference Galois group and its standard representation}

Let $(L,\phi)$ be a PV extension of $(\ref{eq:general diference system}).$
\begin{defn}
The group
\[
G=\mathrm{Aut}_{\phi}(L/K)
\]
of $\phi$-automorphisms of $L$ leaving $K$ pointwise fixed is called
the \emph{(difference) Galois group} of $(\ref{eq:general diference system}).$
\end{defn}

Fix a fundamental matrix $U$ of $(\ref{eq:general diference system})$ so that $L=K(U)$ and every $\sigma\in G$
is determined by its effect on $U$. Since $\sigma$ and $\phi$ commute,
$\sigma(U)$ is another fundamental matrix, so
\[
\sigma(U)=UV_{\sigma}
\]
for $V_{\sigma}\in \GL_{n}(C)$.  \cite[Theorem 1.13]{S-vdP97} shows that $\sigma\mapsto V_{\sigma}$
is an embedding of $G$ as a \emph{Zariski closed} subgroup of $\GL_{n}(C).$
A different choice of $U$ changes the embedding by conjugation.

Another look at $G$ is given by the $\phi$-module $(W,\Phi)=(K^{n},A^{-1}\phi)$
associated with $(\ref{eq:general diference system})$, and the solution
set
\[
\mathcal{U}=W_{L}^{\Phi}=UC^{n}.
\]
The $G$-action on $W_{L}=W\otimes_{K}L$ ($\sigma$ acting like $1\otimes\sigma$)
commutes with $\Phi,$ so induces an action on $\mathcal{U}.$ If
$\sigma\in G$ and $Uv$ ($v\in C^{n})$ is a vector in $\mathcal{U},$
then $\sigma(Uv)=UV_{\sigma}v.$ It follows that the representation
$\sigma\mapsto V_{\sigma}$ is nothing but the matrix representation
afforded by $\mathcal{U},$ in the basis consisting of the columns
of $U$. We call $\mathcal{U}$ the \emph{standard representation}
of $G$.

\subsubsection{The Galois correspondence theorem}

Since the characteristic is $0$, the  algebraic group $G$ must be reduced, but need not
be connected. \cite[Lemma 1.28]{S-vdP97} proves that
\[
L^{G}=K.
\]
The last fact is the basis for the Galois correspondence between $\phi$-pseudofields
$K\subset E\subset L$ and Zariski closed subgroups $\{e\}\subset H\subset G.$
With $E$ we associate
\[
H(E)=\mathrm{Aut}_{\phi}(L/E).
\]
With $H$ we associate its fixed field
\[
E(H)=L^{H}.
\]
The Galois correspondence theorem asserts that these two assignments
are inverse to each other, and set the family of all intermediate
$\phi$-pseudofields and the family of all Zariski closed subgroups
of $G$ in an order-reversing bijection  \cite[Theorem 1.29]{S-vdP97}.

An intermediate $\phi$-pseudofield $E$ is called \emph{normal} if
it is the PV extension of a system $\phi(y)=A_{1}y$ for some $A_{1}\in \GL_{m}(K)$
over $K$.  \cite[Th\'{e}or\`{e}me 3.5.2.2]{An} yields that an intermediate $\phi$-pseudofield  $E$ is normal if and only if the
corresponding group $H=H(E)$ is normal in $G$, and in this case
\[
\mathrm{Aut}_{\phi}(E/K)=G/H.
\]
As in classical Galois theory, when it comes to showing that normal
subextensions of a Galois extension are splitting fields of suitable
polynomials, the system $\phi(y)=A_{1}y$ corresponding to the fixed
field of a normal subgroup of $G$ is neither unique, nor related
in any canonical way to the original system used to define $L$.

An important normal subgroup of $G$ is its connected component $G^{0},$
the smallest Zariski closed subgroup $H$ of finite index in $G$.
Its fixed field $E(G^{0})$ is therefore the largest finite $\phi$-extension
of $K$ in $L.$ If $L$ is a field, this is the algebraic closure
of $K$ in $L$. Thus, if $L$ happens to be a field, $G$ is connected
if and only if $K$ is algebraically closed in $L.$ In this case,
if $K$ does not admit any finite field extension to which $\phi$
extends as an automorphism, $G$ will be connected. The $K$ of \S~\ref{subsec:K phi and partial}
is such an example, see Lemma \ref{lem:no finite phi extensions of K}.

\subsubsection{A special case of the Tannakian correspondence}

As $C$ is algebraically closed of characteristic 0, the algebraic
group $G$ may be identified, as we did, with its $C$-points $G(C).$
It can be defined also as a functor on the category of $C$-algebras
by
\[
G(D)=\mathrm{Aut}_{\phi}(D\otimes_{C}R/D\otimes_{C}K)
\]
where $\phi$ is extended to $D\otimes_{C}R$ as $1\otimes\phi$.
One proves then that this functor is representable by a linear algebraic
group over $C,$ and the standard Yoneda Lemma shows that it determines
$G$ uniquely. A more sophisticated approach is to consider the abelian
category of $\phi$-modules, the  tensor subcategory $\{W\}$ generated
by the $\phi$-module $W$ attached to our system, and use the Tannakian
formalism to obtain $G$ as its Tannakian fundamental group (see \cite[Th\'{e}or\`{e}me 3.4.2.3]{An} for a discussion on the notions of fiber functors on $\{W\}$ and Picard-Vessiot extensions for $W$).

We shall need the following special case of the Tannakian correspondence,
which can be proven directly. We assume that $C$ is algebraically
closed.
\begin{prop}
\label{prop:Tannakian correspondence}Let the notation be as above.
The applications
\[
\xymatrix{ 
\mathcal{V} \subset \mathcal{U}   \ar@{|->}[r] &   V:=(\mathcal{V}\otimes_{C}L)^{G}\subset W  \mbox { and }
 V \subset W     \ar@{|->}[r] & \mathcal{V}:=(V\otimes_{K}L)^{\Phi}\subset\mathcal{U}   }
\]
are bijections between $G$-submodules $\mathcal{V}$ of $\mathcal{U}$,
and $\phi$-submodules $V$ of $W$, which are inverse to each other.
We have $\dim_{K}V=\dim_{C}\mathcal{V}.$ In particular $W$ is an
irreducible $\phi$-module if and only if $\mathcal{U}$ is an irreducible
representation of $G$.
\end{prop}

Here $G$ acts on $\mathcal{V}\otimes_{C}L$ diagonally, and the operator
$\Phi$ on $V$ is derived from $1\otimes\phi$. Similarly, $\Phi$
acts on $V\otimes_{K}L$ diagonally, and the $G$-action on $\mathcal{V}$
is derived from its action on $L$.

In the next section we review the $\delta$-parametrized Picard Vessiot
theory. There too, the three approaches (Picard-Vessiot, schematic and Tannakian)
coexist. The last two, however, have not been fully developed in the
literature\footnote{Despite partial results of Buium, Kamensky, Kovacic and Ovchinnikov.},
so we adopt the Picard-Vessiot approach. As we shall explain, the $\delta$-parametrized
Galois group is a \emph{linear differential algebraic group}, which
is not determined by its points in $C$, but rather in a \emph{differential
closure} $\widetilde{C}$ of $C.$ This forces us to extend scalars
from $C$ to $\widetilde{C}$, and then use descent arguments to go
back.

\subsection{$\delta$-parametrized Picard Vessiot theory of difference equations}

\subsubsection{LDAG's}\label{subsec:LDAG}

Linear differential algebraic groups (LDAG's) have been defined and
studied by Kolchin \cite{Kol} and Cassidy \cite{Cas72}. For a quick
introduction see, for example, the summary in \S2 of \cite{Mi-Ov}.

Let $C$ be an algebraically closed field of characteristic 0, equipped
with a derivation $\delta$, which may be trivial. We fix a differential
closure $(\widetilde{C},\delta)$ of $C$. Let
\[
\widetilde{C}\{X_{ij},\det(X)^{-1}\}_{\delta}
\]
denote the ring of differential polynomials in the variables $X_{ij}$ ($1\le i,j\le n$),
with $\det(X)$ inverted. A LDAG  $\widetilde{\mathcal{G}}$ is a subgroup of $\GL_{n}(\widetilde{C})$ that is the zero set  
of some radical $\delta$-ideal $\widetilde{\mathcal{I}}$ of $\widetilde{C}\{X_{ij},\det(X)^{-1}\}_{\delta}$.
We call $\widetilde{\mathcal{I}}$ the \emph{ideal
of definition of $\widetilde{\mathcal{G}}$}. It  is a  radical Hopf $\delta$-ideal   and by the Ritt-Raudenbusch theorem,  is the radical of a finitely generated $\delta$-ideal.

If  the generators of $\widetilde{\mathcal{I}}$ as a $\delta$-ideal
can be taken to be differential polynomials with coefficients from
$C$, we say that $\widetilde{\mathcal{G}}$ is \emph{defined over $C$}. For such a LDAG, we attach a \emph{differential group scheme} $\mathcal{G}$ over $C$ as follows. We define the ideal of definition of $\mathcal{G}$  as 
$
\mathcal{I}=\widetilde{\mathcal{I}}\cap C\{X_{ij},\det(X)^{-1}\}_{\delta}
$ and we  let
$C\{\mathcal{G}\}=C\{X_{ij},\det(X)^{-1}\}_{\delta}/\mathcal{I}$
be the \emph{differential coordinate ring of $\widetilde{\mathcal{G}}$ over $C$}. Then, we define
the points of $\mathcal{G}$ in any $\delta$-ring  extension $D$
of $C$ as
\[
\mathcal{G}(D)=\mathrm{Hom}_{\delta}(C\{\mathcal{G}\},D).
\]
As $\widetilde{\mathcal{I}}$ is $\delta$-generated by $\mathcal{I}$, we see that $\mathcal{G}(\widetilde{C})=\widetilde{\mathcal{G}}$, $\mathcal{I}$ is a radical Hopf $\delta$-ideal in $C\{X_{ij},\det(X)^{-1}\}_{\delta}$, and $\mathcal{G}(D)$ is a {\it subgroup} of $\GL_{n}(D)$.

Though $\mathcal{G}(\widetilde{C})=\widetilde{\mathcal{G}}$, we caution that the group of $C$-points $\mathcal{G}(C)$ tells us
little about the nature of $\mathcal{G}.$ For example, the single
equation
\[
 \delta(\frac{\delta X}{X})=0
\]
is easily seen to define a differential subgroup $\mathcal{G}$ of
$\GL_{1}(\widetilde{C}).$ If $C$ is the field of complex numbers,
equipped with the trivial derivation, then $\mathcal{G}(\mathbb{C})=\mathbb{C}^{\times}.$
If $\delta$ is extended to $\mathbb{C}(z)$ so that $\delta(z)=1$
we still have $\mathcal{G}(\mathbb{C}(z))=\mathbb{C}^{\times}.$ However,
over the field of meromorphic functions on $\mathbb{C}$ (with $\delta=d/dz$)
we find new points, namely the solutions $X=e^{az}$ for any $a\in\mathbb{C}.$

The \emph{Zariski closure} $\widetilde{G}$ of $\widetilde{\mathcal{G}}$  is a linear algebraic
group defined over $\widetilde{C}$. Its ideal of definition as a linear algebraic group over $\widetilde{C}$
is
\[
\widetilde{I}=\widetilde{\mathcal{I}}\cap \widetilde{C}[X_{ij},\det(X)^{-1}].
\]
Every linear algebraic group $\widetilde{G}\subset \GL_{n}(\widetilde{C})$ may
be considered also a LDAG, which we denote $[\delta]_{*}\widetilde{G}$. If $G$ is a linear algebraic group defined over $C$, then $G(\widetilde{C})$ is a LDAG defined over $C$. Abusing notation, we  denote by  $[\delta]_{*}G$ the differential group scheme over $C$ associated with $G(\widetilde{C})$.   The Zariski
closure $\widetilde{G}$ of $\widetilde{\mathcal{G}}$ is  characterized by the property
that $\widetilde{\mathcal{G}}\subset[\delta]_{*}\widetilde{G}$, and for no proper Zariski
closed subgroup $\widetilde{H}\subsetneq \widetilde{G}$ do we have $\widetilde{\mathcal{G}} \subset[\delta]_{*}\widetilde{H}.$

If $\mathcal{G}$ is a connected differential group scheme over $C$, \emph{i.e.},  $C\{\mathcal{G}\}$
is a domain, then its field of fractions $C\langle \mathcal{G} \rangle$ is a $\delta$-field.
Its $\delta$\emph{-transcendence degree} over $C$, denoted by $\delta\mathrm{tr.deg.}(C\langle \mathcal{G} \rangle/C)$, is the maximal
number of elements of $C\langle \mathcal{G} \rangle$ which are $\delta$-algebraically
independent over $C$, \emph{ i.e.}, do not satisfy any differential polynomial
with coefficients from $C$. It coincides with the $\delta$\emph{-dimension}
of $\mathcal{G}$, as defined by Kolchin (\cite{Kol}, IV.3, p.148):
\[
\delta\dim\mathcal{G}=\delta\mathrm{tr.deg.}(C\langle \mathcal{G} \rangle/C).
\]
In general $\delta\dim \mathcal{G}$ is equal to the 
 $\delta$-dimension
of its connected component (in the Kolchin topology). For example,
the $\delta$-dimension of the example given above is $0$.

A LDAG  $\widetilde{\mathcal{G}}$ is called \emph{$\delta$-constant} if its ideal of definition
$\widetilde{\mathcal{I}}$ contains $\delta(X_{ij})$, or equivalently $\widetilde{\mathcal{G}}\subset \GL_{n}(\widetilde{C}^{\delta}).$
The following theorem of Cassidy (\cite{Cas89}, Theorem 19) is instrumental
to our work.
\begin{thm}
\label{thm: Cassidy's theorem on simple Zariski closure}Suppose that
$\widetilde{\mathcal{G}}$ is a LDAG, Zariski dense in a \emph{simple }linear
algebraic group $\widetilde{G}\subset \GL_{n}(\widetilde{C}).$ If $\widetilde{\mathcal{G}}\varsubsetneq[\delta]_{*}\widetilde{G},$
then $\widetilde{\mathcal{G}}$ is conjugate, in $\GL_{n}(\widetilde{C}),$ to
a $\delta$-constant LDAG.
\end{thm}

\subsubsection{The $\delta$-parametrized Picard Vessiot extension and Picard Vessiot
ring}\label{subsubsec:PPVext}

Let $(K,\phi,\delta)$ be a $(\phi,\delta)$-field of characteristic
0, where $\phi$ is an automorphism and $\delta$ a derivation commuting
with $\phi$: 
\[
\phi\circ\delta=\delta\circ\phi.
\]
We assume that the field $C=K^{\phi}$ of $\phi$-constants is algebraically
closed. Since $\delta$ and $\phi$ commute, it is a $\delta$-field,
and we denote by $\widetilde{C},$ as before, a differential closure,
on which we let $\phi$ act trivially.

Let $(\ref{eq:general diference system})$ be a linear system of difference
equations over $K$.
\begin{defn}
A $\delta$-\emph{parametrized Picard Vessiot (PPV) extension} of
$K$ is a $(\phi,\delta)$-pseudofield $\mathcal{L}$, \emph{i.e.},  \[\mathcal{L}=\mathcal{L}_{1}\times\cdots\times \mathcal{L}_{r},
\]
where the  $\mathcal{L}_i$ are $\delta$-field extensions of $K$ permuted cyclically by $\phi$, satisfying:

\begin{itemize}
\item $\mathcal{L}=K\left\langle U\right\rangle _{\delta}$ is $\delta$-generated
as a pseudofield by the entries of a fundamental matrix $U$ for the
system.
\item $\mathcal{L}^{\phi}=C.$
\end{itemize}
\end{defn}
By \cite{Wi12}, Corollary 10, a $\delta$-parametrized Picard Vessiot
extension exists. By Proposition 6.16 of \cite{H-S08}, if $C=\widetilde{C}$,
it is also unique up to $K$-$(\phi,\delta)$-isomorphism.

The subring $\mathcal{R}=K\{U,\det(U)^{-1}\}_{\delta}$ of $\mathcal{L}$ that is $\delta$-generated
over $K$ (as a $\delta$-ring) by $U$ and the inverse of its determinant,
does not depend on the choice of $U$ and is $\phi$-simple \cite[Prop.2.14]{O-W}. Such
a ring is called a $\delta$-\emph{parametrized Picard Vessiot ring}
for the system\footnote{In \cite{H-S08} one only asks that $\mathcal{R}$ be $(\phi,\delta)$-simple,
a weaker condition, but it is shown in Corollary 6.22 that if $C$
is differentially closed, $\mathcal{R}$ is then actually $\phi$-simple.
See the discussion in \cite{Wi12} why, when $C$ is only algebraically
closed, it makes better sense to impose the stronger condition of
being $\phi$-simple. }. The subring $R=K[U,\det(U)^{-1}]$ of $\mathcal{R}$ is an (ordinary)
PV ring then, and its total ring of fractions $L\subset\mathcal{L}$
an (ordinary) PV extension.

\subsubsection{The $\delta$-parametrized Galois group}\label{subsec:The--parametrized-Galois}
Let $(K,\phi,\delta)$ be a $(\phi,\delta)$-field of characteristic zero with algebraically closed field 
of $\phi$-constants $C$.

For $\mathcal{L}$ a PPV extension of \eqref{eq:general diference system} over $K$ and $\mathcal{R} \subset \mathcal{L}$ the PPV-ring, we define the \emph{$\delta$-parametrized Galois group scheme of $\mathcal{L}$ over $K$ } as the functor which associates to any  $\delta$-ring extension $D$ of $C$, the group
\[
\mathcal{G}(D)=\mathrm{Aut}_{\phi,\delta}(D\otimes_{C}\mathcal{R}/D\otimes_{C}K).
\]
This functor is indeed representable by a differential group scheme over $C$, which is unique up to isomorphism by the Yoneda Lemma.

If $\widetilde{C}$ is a differential closure of $C$, one can consider $\widetilde{K}=Frac(\widetilde{C}\otimes_{C}K)$, the base-changed
ground field ($\widetilde{C}\otimes_{C}K$ is a domain because $\widetilde{C}$
is $C$-regular), and $\mathfrak{S}=\widetilde{C}\otimes_{C}K-\{0\}$
(a multiplicative set). Let 
$\widetilde{\mathcal{R}}=\widetilde{C}\otimes_{C}\mathcal{R}[\mathfrak{S}^{-1}]=\widetilde{K}\otimes_{K}\mathcal{R}.$
We claim that $\widetilde{\mathcal{R}}$ is a PPV ring over $\widetilde{K}.$ 
For that we only have to show that it is $\phi$-simple, and as it
is obtained by localization from $\widetilde{C}\otimes_{C}\mathcal{R}$
it is enough to check that the latter is $\phi$-simple. But this
is clear, since the action of $\phi$ on $\widetilde{C}$ is trivial and $\mathcal{R}$ is $\phi$-simple \cite[Lemma 1.11]{S-vdP97}.
Let $\mathcal{\widetilde{L}}$ be the total ring of fractions of $\widetilde{\mathcal{R}}.$ It is a
a PPV extension over $\widetilde{K}$ with $\widetilde{\mathcal{L}}^{\phi}=\widetilde{C}$ \cite[Cor. 6.15]{H-S08}.

Since any $(\phi,\delta)$-automorphism $\sigma$ of $\mathcal{\widetilde{L}}$
over $\widetilde{K}$ is determined by its effect on $U$, and $\sigma(U)=UV_{\sigma}$
where $V_{\sigma}$ has entries in $\widetilde{\mathcal{L}}^{\phi}=\widetilde{C},$
such a $\sigma$ actually induces an automorphism of $\widetilde{C}\otimes_{C}\mathcal{R}$
over $\widetilde{C}\otimes_{C}K.$ The converse is equally clear. 
We define the \emph{$\delta$-parametrized Galois group } of $\widetilde{\mathcal{L}}$ over $\widetilde{K}$ as $\mathrm{Aut}_{\phi,\delta}(\widetilde{\mathcal{L}}/\widetilde{K})$. Since $\widetilde{C}$ is $\delta$-closed,  this group can be embedded as a LDAG $\widetilde{\mathcal{G}}$ in $\GL_n(\widetilde{C})$ via its action on a fundamental matrix in $\GL_n(\widetilde{R})$ and it is independent, up to conjugation in $\GL_n(\widetilde{C})$, of the choice of the PPV extension over $\widetilde{K}$ or of $U$ \cite[Proposition 6.18]{H-S08}. Finally, 
\cite[Proposition 1.20]{dV-H} yields
\[
\widetilde{\mathcal{G}}=\mathrm{Aut}_{\phi,\delta}(\widetilde{\mathcal{L}}/\widetilde{K})= \mathcal{G}(\widetilde{C}).
\]

Let $\widetilde{G}=\mathrm{Aut}_{\phi}(\widetilde{L}/\widetilde{K})$ be the (ordinary) difference Galois
group of \eqref{eq:general diference system} over $\widetilde{K}$. Then the group $\widetilde{\mathcal{G}}$ is Zariski
dense in $\widetilde{G}$,  see \cite{H-S08}, Proposition 6.21.

Finally,  the  torsor theorem yields the following equality
\[ \delta \dim \mathcal{G} = \delta \mathrm{tr.deg.}( \mathcal{L}/K),\]
where $\delta \mathrm{trdeg}(\mathcal{L}/K)$ is the differential transcendence degree of $\mathcal{L}_1$ over $K$ in the notation of \S \ref{subsubsec:PPVext}. This result was proved for the $\delta$-parametrized Galois group $\widetilde{\mathcal{G}}$ in \cite[Proposition 6.26]{H-S08}. For the $\delta$-parametrized Galois group scheme $\mathcal{G}$, the proof of the above equation is entirely analogous to the proof of \cite[Lemma 2.7]{DVHW} in the symmetric context of differential equations with a difference parameter.

\subsubsection{The $\delta$-parametrized Galois correspondence}

We shall need the following result. The proof of Lemma 6.19 in \cite{H-S08},
although set in a different context, applies here as well, without
any change.
\begin{prop} Let the notation be as in the previous  section.
For every $x\in\widetilde{\mathcal{L}}-\widetilde{K}$ there exists
a $\sigma\in\widetilde{\mathcal{G}}$ with $\sigma(x)\ne x,$ \emph{i.e.},
\[
\widetilde{\mathcal{L}}^{\widetilde{\mathcal{G}}}=\widetilde{K}.
\]
\end{prop}

This is the key to the $\delta$-parametrized Galois correspondence,
analogous to what we described in the non-parametrized framework.
See \cite{H-S08} and note that since we do not work schematically,
we must extend scalars to $\widetilde{C}.$

\subsection{$\delta$-algebraic solutions\label{subsec:delta-algebraic-solutions}}

Notation as in section \ref{subsec:The--parametrized-Galois}, let $(W,\Phi)=(K^{n},A^{-1}\phi)$
be the $\phi$-module associated with the linear system $\phi(y)=Ay,$
$L$ a PV extension over $K$ and $\mathcal{L}$ a $\delta$-parametrized PV
extension over $K$ containing $L$. Let
\[
\mathcal{U}=W_{L}^{\Phi}=W_{\mathcal{L}}^{\Phi}=UC^{n}
\]
be the solution space, where $U\in \GL_{n}(L)$ is a fundamental matrix.
As above, we denote by $\widetilde{C}$ a differential closure of
$C$, and add a tilde to denote the same objects over $\widetilde{K}$.
Since $\widetilde{\mathcal{L}}^{\phi}=\widetilde{C},$ we have $\widetilde{\mathcal{U}}=\widetilde{C}\otimes_{C}\mathcal{U}.$

Let $\mathcal{L}_{a}\subset\mathcal{L}$ be the set of elements that
are $\delta$-algebraic over $K$. Since $x\in\mathcal{L}_{a}$ if
and only if
\[
\mathrm{tr.deg.}K(x,\delta x,\delta^{2}x,...)/K<\infty,
\]
it is clear that $\mathcal{L}_{a}$ is a $\delta$\emph{-invariant
subfield} of $\mathcal{L}$. Since $\phi$ and $\delta$ commute,
it is also $\phi$\emph{-invariant}. Let
\[
\mathcal{U}_{a}=\mathcal{U}\cap\mathcal{L}_{a}^{n}
\]
be the $C$-subspace of $\mathcal{U}$ consisting of solutions all
of whose coordinates are $\delta$-algebraic. Similarly, define $\mathcal{\widetilde{L}}_{a}\subset\mathcal{\widetilde{L}}$
to be the field of elements that are $\delta$-algebraic over $\widetilde{K}$,
and $\mathcal{\widetilde{U}}_{a}=\mathcal{\widetilde{U}}\cap\mathcal{\widetilde{L}}_{a}^{n}$.

If $\sigma\in\widetilde{\mathcal{G}}=\mathcal{G}(\widetilde{C})=\mathrm{Aut}_{\phi,\delta}(\widetilde{\mathcal{L}}/\widetilde{K})$
and $x\in\mathcal{\widetilde{L}}_{a}$, then $\sigma(x)\in\mathcal{\widetilde{L}}_{a}$
because $\sigma$ commutes with $\delta.$ Thus $\widetilde{\mathcal{G}}$
preserves $\mathcal{\widetilde{L}}_{a}$, and in its standard representation
on $\widetilde{\mathcal{U}},$ $\widetilde{\mathcal{U}}_{a}$ becomes
an invariant subspace.
\begin{lem}
\label{lem: base change of delta algebraic solutions}We have $\mathcal{\widetilde{U}}_{a}=\widetilde{C}\otimes_{C}\mathcal{U}_{a}.$
\end{lem}

\begin{proof}
Let $\mathcal{R}$ be the $\delta$-parametrized PV ring in $\mathcal{L}$.
As $\mathcal{U}\subset\mathcal{R}^{n},$ $\widetilde{\mathcal{U}}\subset\widetilde{C}\otimes_{C}\mathcal{R}^{n}$
and it is enough to prove that
\begin{equation}\label{tensor Ctilde commutes with a}
(\widetilde{C}\otimes_{C}\mathcal{R})_{a}=\widetilde{C}\otimes_{C}\mathcal{R}_{a}. 
\end{equation}

The proof of \eqref{tensor Ctilde commutes with a} is done precisely as in \cite{A-D-H-W}, Lemma A.15, where the
same statement is proved if the derivation $\delta$ is replaced by
a difference operator (\emph{i.e.} a field automorphism, denoted there $\sigma$).
The only non-formal fact used in the proof of that Lemma would be,
in our context, the statement that for any $x\in\widetilde{C}-C$
there exists a $\delta$-automorphism of $\widetilde{C}$ over $C$
not fixing $x.$ For this, see the \emph{proof }of Proposition 1.1
in \cite{M-RDCF}.
\end{proof}
\begin{cor}
\label{cor:Galois invariance of delta algebraic solutions}The $C$-subspace
$\mathcal{U}_{a}\subset\mathcal{U}$ is $G$-invariant.
\end{cor}

\begin{proof}
As $\mathcal{\widetilde{U}}_{a}$ is $\widetilde{\mathcal{G}}$-invariant
and $\widetilde{\mathcal{G}}$ is Zariski dense in $\widetilde{G}$,
$\mathcal{\widetilde{U}}_{a}$ is $\widetilde{G}$-invariant. But
the subspace $\mathcal{U}_{a}$ and the algebraic group $G\subset \GL(\mathcal{U})$
are defined over $C$, $\widetilde{\mathcal{U}}_{a}$ is nothing but
the $\widetilde{C}$-points of $\mathcal{U}_{a}$ (by the last lemma)
and $\widetilde{G}=G(\widetilde{C}).$ Thus $\mathcal{U}_{a}$ is
$G$-invariant.
\end{proof}

\section{\label{sec:A-Galoisian-criterion}A Galoisian criterion for $\delta$-integrability}

Let $(K,\phi,\delta)$ be as above ($\phi$ and $\delta$ commuting
with each other, $C=K^{\phi}$ algebraically closed). Recall that
$(\ref{eq:general diference system})$ is called $\delta$-integrable
if the associated $\phi$-module $(W,\Phi)$ carries a compatible
connection $\nabla$ making $(W,\Phi,\nabla)$ a $(\phi,\delta)$-module
over $K.$ Equivalently, integrability means that there exists a matrix
$B\in\mathfrak{gl}_{n}(K)$ such that
\begin{equation}
\delta(A)=\phi(B)A-AB.\label{eq:integrability relation-1}
\end{equation}

\begin{prop}
\label{prop:Galois criterion for integrability-1} The system $(\ref{eq:general diference system})$
is $\delta$-integrable if and only if there exists a matrix $D\in\mathfrak{gl}_{n}(C)$
such that
\begin{equation}\label{eq:integrabilitycondition group}
\delta(V_{\sigma})=V_{\sigma}D-DV_{\sigma}
\end{equation}
for every $\sigma\in\mathcal{G}.$
\end{prop}

To be precise, the meaning of the criterion is the following. Fix a fundamental solution matrix $U$ with coefficients in $\mathcal{R}$. For
any $\delta$-ring extension $C\subset C'$, and for any $\sigma\in\mathcal{G}(C'),$ the matrix $V_{\sigma} \in \GL_n(C')$  such that $\sigma(U)=UV_{\sigma}$ satisfies 
 \eqref{eq:integrabilitycondition group} in $\mathfrak{gl}_{n}(C').$ Alternatively,
if $\mathcal{I}\subset C\{X_{ij},\det(X)^{-1}\}_{\delta}$ is the
ideal of definition of $\mathcal{G},$ then for any $1\le i,j\le n$
\[
\delta(X_{ij})-\sum_{\ell}X_{i\ell}D_{\ell j}+\sum_{\ell}D_{i\ell}X_{\ell j}\in\mathcal{I}.
\]

\begin{proof}
Suppose $A$ and $B$ satisfy $(\ref{eq:integrability relation-1}).$
Let $\mathcal{R}\subset\mathcal{L}$ be the $\delta$-parametrized
PV ring and extension, and $U\in \GL_{n}(\mathcal{R})$ a fundamental
matrix. Then
\[
\phi(\delta(U)-BU)=\delta(AU)-\phi(B)AU=A(\delta(U)-BU),
\]
so
\[
D:=U^{-1}(\delta(U)-BU)\in\mathfrak{gl}_{n}(C),
\]
as it is fixed by $\phi$ and $\mathcal{L}^{\phi}=C.$ Calculating
$\delta(\sigma U)=\sigma(\delta U)$ for $\sigma\in\mathcal{G}(C'),$
$C'$ as above, we find
\[
\delta(U)V_{\sigma}+U\delta(V_{\sigma})=\delta(UV_{\sigma})=\delta(\sigma U)=\sigma(\delta U)=\sigma(BU+UD)=BUV_{\sigma}+UV_{\sigma}D,
\]
or
\[
\delta(V_{\sigma})=V_{\sigma}D-DV_{\sigma}.
\]

Conversely, if $D\in\mathfrak{gl}_{n}(C)$ satisfies the last equation,
define
\[
B:=\delta(U)U^{-1}-UDU^{-1}\in\mathfrak{gl}_{n}(\mathcal{R})\subset\mathfrak{gl}_{n}(\mathcal{L})\subset\mathfrak{gl}_{n}(\mathcal{\widetilde{L}}).
\]
Then, for $\sigma\in\mathcal{G}(\widetilde{C})$
\[
\sigma(B)=\delta(UV_{\sigma})V_{\sigma}^{-1}U^{-1}-UV_{\sigma}DV_{\sigma}^{-1}U^{-1}=
\]
\[
=\delta(U)U^{-1}+U(V_{\sigma}D-DV_{\sigma})V_{\sigma}^{-1}U^{-1}-UV_{\sigma}DV_{\sigma}^{-1}U^{-1}=B,
\]
so $B\in\mathfrak{gl}_{n}(\widetilde{K})$ by the $\delta$-parametrized
Galois correspondence: $\mathcal{\widetilde{L}}^{\mathcal{G}}=\widetilde{K}.$
But $\mathcal{R}\cap\widetilde{K}=K$ so we can descend the field
of $\phi$-scalars and deduce that $B\in\mathfrak{gl}_{n}(K).$ We
compute
\[
\phi(B)A-AB=\delta(AU)U^{-1}-AUDU^{-1}-AB
\]
\[
=\delta(A)+A\delta(U)U^{-1}+A(B-\delta(U)U^{-1})-AB=\delta(A).
\]
Thus $A$ satisfies the condition for $\delta$-integrability.
\end{proof}
\begin{rem*}
\begin{enumerate}[i)]

\item The above Proposition does not require $K^{\phi}=C$ to be differentially
closed.

\item Compare \cite{H-S08}, Proposition 2.9. Assuming $C$ is differentially
closed, the relation $\delta(V_{\sigma})=[V_{\sigma},D]$ is \emph{integrated}
there to conclude that $\mathcal{G}$ is \emph{conjugate} to a $\delta$-constant
group (see Section \ref{subsec:LDAG}). If $C$ is not differentially
closed, such a conjugation exists only over a suitable non-trivial
$\delta$-extension of $C,$ as we need to find an invertible matrix
$E$ solving $\delta E=-DE,$ \emph{ i.e.}, a fundamental matrix for $\delta y=-Dy.$
Having such an $E$ at hand,
\[
\delta(E^{-1}V_{\sigma}E)=-E^{-1}\delta(E)E^{-1}V_{\sigma}E+E^{-1}\delta(V_{\sigma})E+E^{-1}V_{\sigma}\delta(E)
\]
\[
=E^{-1}(DV_{\sigma}-V_{\sigma}D+\delta(V_{\sigma}))E=0.
\]
\end{enumerate}

\end{rem*}

\section{$\partial$-modules over $M$ and monodromy}

\subsection{Triviality of $\partial$-modules over $M$}

From now on:
\begin{itemize}
\item $K=\bigcup_{\Lambda\in\mathfrak{L}}K_{\Lambda}$, $\phi f(z)=f(qz)$
and $\partial f(z)=f'(z)$ are as in \S~\ref{subsec:K phi and partial}.
\item $M=\mathscr{M}(\mathbb{C})$ is the field of meromorphic functions
on $\mathbb{C}$, with the same $\partial,$$\phi.$
\item $F=\mathbb{C}((z)),$ $\mathcal{O}_{F}=\mathbb{C}[[z]],$ same $\partial,\phi.$
\item $K\subset M\subset F,$ as $(\phi,\partial)$-fields.
\end{itemize}
Let $B\in\mathfrak{gl}_{n}(M).$ For any $\zeta\in\mathbb{C}$ let
$M_{\zeta}$ be the field of germs of meromorphic functions at $\zeta$,
and $\mathcal{O}_{\zeta}$ its valuation ring. Consider the system
\begin{equation}
\partial y=By\label{eq:differential system-1}
\end{equation}
and the associated $\partial$-module $W=M^{n},$ $\nabla(y)=\partial y-By.$
\begin{lem}
The following are equivalent:
\begin{enumerate}[i)]
\item  The system (\ref{eq:differential system-1}) has a full set of
solutions in $M_{\zeta}.$

\item  The $\partial$-module $W_{\zeta}=M_{\zeta}\otimes_{M}W$ is
trivial.

\item  $B$ has an apparent singularity at $\zeta,$ \emph{i.e.}, there exists
a gauge transformation
\[
\widetilde{B}=\partial P\cdot P^{-1}+PBP^{-1}
\]
with $P\in \GL_{n}(M_{\zeta})$ such that $\widetilde{B}$ is regular
(holomorphic) at $\zeta.$ \end{enumerate}
\end{lem}

\begin{proof}
(i) is equivalent to $\dim_{\mathbb{C}}W_{\zeta}^{\nabla}=n.$ If
this is the case then by the Wronskian Lemma $W_{\zeta}=W_{\zeta}^{\nabla}\otimes_{\mathbb{C}}M_{\zeta}.$
This proves that (i) implies (ii). If (ii) holds then there exists even a $P$ with
$\widetilde{B}=0,$ whence (iii). Note that in this case, since $\partial(P^{-1})=-P^{-1}\partial(P)P^{-1}$
we have $\partial(P^{-1})=BP^{-1}$ so $P^{-1}$ is a fundamental
matrix of solutions in $M_{\zeta}.$ Finally, if (iii) holds then
to get (i) we may assume, by a change of coordinates, that $B$ was
regular at $\zeta$ to begin with. By the basic existence and uniqueness
theorem for linear systems of ordinary differential equations, we
can find a full set of solutions in $M_{\zeta}.$
\end{proof}

According to the above lemma, we say that a $\partial$-module $W$ has apparent singularities if all the singularities of an associated differential system are apparent.
\begin{cor}
\label{cor:apparent singularities and =00005Cpartial triviality-1}
A $\partial$-module $W$ over $M$ has apparent singularities if and only if it is trivial.
\end{cor}

\begin{proof}
Since all the singularities are apparent, one can continue solutions
meromorphically along paths indefinitely. Since $\mathbb{C}$ is simply
connected, this yields a single-valued solution for any initial conditions
at $0.$  
\end{proof}

\subsection{Periodicity and monodromy}

Suppose now that $B\in\mathfrak{gl}_{n}(K_{\Lambda})$ and all the
singularities of $(\ref{eq:differential system-1})$ are apparent.
Let $U$ be a fundamental matrix for $(\ref{eq:differential system-1})$
in $M$. If $\omega\in\Lambda$ then $U(z+\omega)$ also satisfies
(\ref{eq:differential system-1}) so there exists a matrix $M(\omega)\in \GL_{n}(\mathbb{C})$
such that
\begin{equation}
U(z+\omega)=U(z)M(\omega).\label{eq:monodromy-1}
\end{equation}
The map $\omega\mapsto M(\omega)$ is a homomorphism $\Lambda\to \GL_{n}(\mathbb{C}),$
the \emph{monodromy} \emph{representation}. Since $\Lambda$ is abelian,
its image, the \emph{monodromy group}, is an abelian subgroup of $\GL_{n}(\mathbb{C}).$

If we replace $U(z)$ by another fundamental matrix $U(z)T$, with
$T\in \GL_{n}(\mathbb{C})$, then the monodromy representation undergoes
conjugation: $\omega\mapsto T^{-1}M(\omega)T.$ Thus intrinsically,
the monodromy representation is well-defined only up to conjugation.
\begin{lem}
\label{lem:Field of definition of Y-1}Let $Z\in \GL_{n}(M)$ be another
matrix of meromorphic functions satisfying $(\ref{eq:monodromy-1})$.
Then $U(z)=Q(z)Z(z)$ for a matrix $Q(z)\in \GL_{n}(K_{\Lambda}).$
If $S$ is a $K_{\Lambda}$-subalgebra of $M$ containing the entries
of $Z$ then the entries of $U$ are in $S$ as well.
\end{lem}

\begin{proof}
$Q(z)=U(z)Z(z)^{-1}$ is invariant under translation by $\Lambda$.
It follows that its entries are $\Lambda$-elliptic.
\end{proof}

\subsection{Consequences of $\phi$-isomonodromy}

Assume that $W$ is a $(\phi,\partial)$-module over $K$, so that
$W$ is a $\phi$-isomonodromic $\partial$-module. Fix a basis of
$W$ over $K$ and identify it with $K^{n},$ where $\Phi$ and $\nabla$
are given by matrices $A$ and $B$ as above.
\begin{lem}
All the singularities of $(W,\nabla)$ are apparent, hence $W$ becomes
$\partial$-trivial over $M$.
\end{lem}

\begin{proof}
$W$ is defined over some $K_{\Lambda}.$ Since $W^{(\phi)}$ and
$W$ are isomorphic, if $\zeta$ is a regular point of $W$ (\emph{i.e.}
an apparent singularity of the system $(\ref{eq:differential system-1})$),
so is $q\zeta.$ For some $\varepsilon>0$, every $\zeta$ in the
punctured disk $0<|\zeta|<\varepsilon$ is regular, so we deduce that
every $\zeta\ne0$ is regular. But then $\zeta=0$ is regular too,
because any $0\ne\omega\in\Lambda$ is regular and $B$ is $\Lambda$-periodic.
The Lemma follows from Corollary \ref{cor:apparent singularities and =00005Cpartial triviality-1}.
\end{proof}
By the discussion in the previous subsection we may associate with
$W$ a monodromy representation $\omega\mapsto M(\omega)$ of $\Lambda$,
well-defined up to conjugation by $\GL_{n}(\mathbb{C}).$
\begin{prop}
The monodromy representation $\omega\mapsto M(\omega)$ is potentially
unipotent: there exists a sublattice $\Lambda'\subset\Lambda$ such
that $M(\omega)$ is unipotent for every $\omega\in\Lambda'$.
\end{prop}

\begin{proof}
Let $U(z)\in \GL_{n}(M)$ be a fundamental matrix for $(W_{M},\nabla)$
and $U^{(\phi)}(z)=U(qz)$ the corresponding fundamental matrix for
$(W_{M}^{(\phi)},\nabla^{(\phi)}).$ Let $\iota:W^{(\phi)}\simeq W$
be an isomorphism of $\partial$-modules, explicitly
\[
\iota(y)=A^{-1}y
\]
where $A\in \GL_{n}(K)$ satisfies the consistency condition $q\phi(B)=ABA^{-1}+\partial(A)A^{-1}.$
Since $\iota\circ\nabla^{(\phi)}=\nabla\circ\iota$, the homomorphism
$\iota$ maps $U^{(\phi)}$ to a fundamental matrix for $W$, so for
some $C\in \GL_{n}(\mathbb{C})$ we have
\[
A(z)^{-1}U(qz)=U(z)C.
\]
Let $\Lambda$ be a lattice of periodicity for $A$ and $B$. Substituting
$z+\omega$ ($\omega\in\Lambda)$ for $z$ we get
\[
A(z)^{-1}U(qz)M(\omega)^{q}=U(z)M(\omega)C=A(z)^{-1}U(qz)C^{-1}M(\omega)C.
\]
It follows that
\[
M(\omega)=CM(\omega)^{q}C^{-1}
\]
and that the set of eigenvalues of $M(\omega)$ is invariant under
raising to power $q$. Since there are only finitely many eigenvalues,
these eigenvalues must be roots of unity, so $M(\omega)$ is potentially
unipotent. Since the monodromy representation is determined by the
(commuting) matrices $M(\omega_{1})$ and $M(\omega_{2}),$ the Proposition
follows.
\end{proof}

\section{$\partial$-integrability and solvability}

The last Proposition has a consequence for the solvability of the
\emph{difference Galois group} of a $\partial$-integrable system
\begin{equation}
\phi(y)=Ay.\label{eq:difference system equation-1}
\end{equation}

\begin{defn}
A $\phi$-module $(W,\Phi)$ over $K$ is \emph{solvable} if there
exists a filtration
\[
0=W_{0}\subset W_{1}\subset W_{2}\subset\cdots\subset W_{n}=W
\]
by $\phi$-submodules such that $\mathrm{dim}_{K}W_{i}=i.$
\end{defn}

If $(W,\Phi)$ is associated to the system $(\ref{eq:general diference system})$
then it is solvable if and only if $A$ is gauge equivalent to an
upper-triangular matrix.
\begin{thm}
If $(W,\Phi)$ is $\partial$-integrable, it is solvable. In fact,
if $\nabla$ is a compatible $\partial$-connection on $W$, a filtration
as above exists where each $W_{i}$ is a $(\phi,\partial)$-submodule.
\end{thm}

\begin{proof}
It is enough to find $W_{1},$ because then we can apply induction
on the dimension to $W/W_{1}$ and lift the filtration found there
to complete the filtration of $W$.

Let $(W,\Phi,\nabla)$ be a  $(\phi,\partial)$-module structure on $W$ as in Lemma~\ref{lem:partialintmoduldestructure}. 
Fix matrices $A$ and $B$ as above, with respect to some basis of
$W$ over $K$. Let $\Lambda$ be a lattice so that all the entries
of $A$ and $B$ lie in $K_{\Lambda},$ \emph{i.e.}, are $\Lambda$-elliptic.
Regarding $(W,\nabla)$ as a $\phi$-isomonodromic $\partial$-module,
and replacing $\Lambda$ by a smaller lattice if necessary, we conclude,
by the previous section, that the monodromy representation
\[ \Lambda \rightarrow\GL_n(\mathbb{C}), \quad \omega \mapsto M(\omega),\]
attached to $(W,\nabla)$ is unipotent, and that $(W,\nabla)$ has
a fundamental matrix $U\in \GL_{n}(M).$ Let $T\in \GL_{n}(\mathbb{C})$
be such that $TM(\omega)T^{-1}$ are all upper-triangular with $1$'s
along the diagonal. Replacing the fundamental matrix $U\in \GL_{n}(M)$
by $UT^{-1}$, $M(\omega)$ is replaced by $TM(\omega)T^{-1}.$ We
may therefore assume, without loss of generality, that $M(\omega)$
are already upper-triangular unipotent. In particular, the first column
of $U$ is a column vector 
\[
u\in K_{\Lambda}^{n}\subset M^{n},
\]
because it satisfies $u(z+\omega)=u(z)$ for all $\omega\in\Lambda.$
But the column vectors of $U$ form a basis of $W_{M}^{\nabla},$
and the first column is, as we have just seen, in $W.$ It follows
that $W^{\nabla}$ is non-zero.

The relation $q\Phi\circ\nabla=\nabla\circ\Phi$ implies that the
$\mathbb{C}$-space $W^{\nabla}$ is $\Phi$-invariant, hence there
exists an eigenvector $e_{1}\in W^{\nabla}$ for $\Phi$. The 1-dimensional
subspace $W_{1}=Ke_{1}$ is the desired $(\Phi,\nabla)$-submodule.
\end{proof}
\begin{cor}
\label{cor: integrability implies solvability-1}Let $G$ be the difference
Galois group of the system
\[
\phi(y)=Ay
\]
over $K$. Assume that $(W,\Phi)$ is $\partial$-integrable. Then
$G$ is solvable.
\end{cor}

\begin{proof}
This is a direct consequence of the Tannakian correspondence, Proposition
\ref{prop:Tannakian correspondence}, see also \cite{S-vdP97}, Theorem
I.1.21. Since $W$ is a solvable $\phi$-module, the standard representation
of $G$ on the solution space $\mathcal{U}$ is solvable too, and
$G$ is contained in the Borel subgroup of upper triangular matrices.
\end{proof}

\section{$\partial$-Triviality over the ring $S_{0}$}

Let $A\in \GL_{n}(K_{\Lambda}),$ $B\in\mathfrak{gl}_{n}(K_{\Lambda})$
be a consistent pair of matrices, and $W$ the corresponding $(\phi,\partial)$-module.
As we have seen, the system 
\[
\partial(y)=By
\]
has only apparent singularities and, hence, it
has a complete set of solutions over $M.$ Let $U\in \GL_{n}(M)$ be
a fundamental matrix of solutions, and $\omega\mapsto M(\omega)$
its monodromy representation. As we have also seen, replacing $\Lambda$
by a sublattice, we may assume that $M(\omega)$ are unipotent for
all $\omega.$ Let $\omega_{1},\omega_{2}$ be an oriented basis of
$\Lambda$ and $M_{i}=M(\omega_{i}).$
\begin{thm}
\label{thm:Entries of U are in S_0}The entries of $U(z)$ are in
the ring $S_{0}=K[z,\zeta(z,\Lambda)].$
\end{thm}

\begin{proof}
According to Lemma \ref{lem:Field of definition of Y-1} we only have
to exhibit a matrix $Z(z)\in \GL_{n}(S_{0})$ with the same monodromy
as $U(z),$ \emph{i.e.}, satisfying
\begin{equation}
Z(z+\omega_{i})=Z(z)M_{i}.\label{eq:monodromy of Z-1}
\end{equation}

Without loss of generality we may assume that $\omega_{2}=1$ and
$\omega_{1}=\tau\in\mathfrak{H}$ (the upper half plane). Since $M_{2}$
is unipotent, 
\[
N_{2}=\log(M_{2})=\sum_{k=1}^{n}\frac{(-1)^{k-1}}{k}(M_{2}-1)^{k}
\]
is a nilpotent matrix satisfying $\exp(N_{2})=M_{2}.$ Note that $N_{2}$
commutes with $M_{1}$, because $M_{1}$ and $M_{2}$ commute. The
matrix $Z(z)$ satisfies $(\ref{eq:monodromy of Z-1})$ if and only
if $\widetilde{Z}(z)=Z(z)\exp(-zN_{2})$ satisfies
\begin{equation}
\widetilde{Z}(z+1)=\widetilde{Z}(z),\,\,\,\widetilde{Z}(z+\tau)=\widetilde{Z}(z)V\label{eq:monodromy of tilde Z-1}
\end{equation}
with $V=M_{1}\exp(-\tau N_{2}).$ Since the entries of $\exp(-zN_{2})$
are polynomials, hence lie in $S_{0},$ it is enough to find $\widetilde{Z}(z)\in \GL_{n}(S_{0})$
satisfying $(\ref{eq:monodromy of tilde Z-1}).$ Note that $V$ is also
unipotent, because $M_{1}$ and $N_{2}$ commute, so $\log(V)$ is
nilpotent.

Let $\ell\in S_{0}$ satisfy $\ell(z+1)=\ell(z)$ and $\ell(z+\tau)=\ell(z)+1.$
By Lemma \ref{lem: real any chi}, such an $\ell$ can be taken to be a $\mathbb{C}$-linear combination of $z$ and
$\zeta(z,\Lambda).$ We now set
\[
\widetilde{Z}(z)=\exp(\ell(z)\log(V)).
\]
We have $\widetilde{Z}\in \GL_{n}(S_{0}),$ because its entries are
polynomials in $\ell(z).$ In addition, $(\ref{eq:monodromy of tilde Z-1})$
is satisfied. This concludes the proof.
\end{proof}
\begin{cor}
\label{cor:Little theorem-1}Suppose $f\in F$ satisfies a linear
homogeneous diferential equation with coefficients from $K,$ as well as a linear homogeneous
$\phi$-difference equation over $K.$ Then $f\in S_{0}$.
\end{cor}

\begin{proof}
Let $W\subset F$ be the $(\phi,\partial)$-module spanned by $f,$
as in Example~\ref{exa:basic example-1}, and let $e_{1},\dots,e_{n}$
be a basis of $W$ over $K$, with $e_{1}=f.$ Let $A^{*}=(a_{ij})\in \GL_{n}(K)$
and $B^{*}=(b_{ij})\in\mathfrak{gl}_{n}(K)$ be defined by
\[
\phi e_{i}=\sum_{j}a_{ij}e_{j},\,\,\,\partial e_{i}=\sum_{j}b_{ij}e_{j}.
\]
These are \emph{not} the matrices $A$ and $B$ that were associated
to $W$ before, but rather $^{t}A^{-1}$ and $-\,^{t}B$, and they
define the \emph{dual }$(\phi,\partial)$-module $W^{*}.$ It is easily
checked that $A^{*}$ and $B^{*}$ are consistent. This follows formally
from the consistency of $(A,B),$ or by duality if we think of the
pair $(A^{*},B^{*})$ as the pair associated with $W^{*}.$ The last
Theorem, applied to $(A^{*},B^{*})$, shows that the equation $\partial y=B^{*}y$
has a full set of solutions in $S_{0},$ \emph{i.e.}, there exists a matrix
$U\in \GL_{n}(S_{0})$ satisfying $\partial U=B^{*}U.$ In \emph{any
$\partial$-field extension }of $S_{0}$\emph{, whose constants are
still $\mathbb{C},$} we find the same solution set for the last system,
namely the $\mathbb{C}$-linear combinations of the columns of $U$.
Since $u=\,{}^{t}(e_{1},\dots,e_{n})\in F^{n}$ is such a solution,
the $e_{i},$ and in particular $f=e_{1}$, lie in $S_{0}$.
\end{proof}

\section{Galois theory and algebraic independence, rank one}

\subsection{The desired theorems}

Our ultimate goal is to prove the following stengthening of Corollary
\ref{cor:Little theorem-1}. Recall that $f\in F$ is called $\partial$\emph{-algebraic}
over $K$ if it satisfies an equation
\[
P(f,\partial(f),\partial^{2}(f),...,\partial^{r}(f))=0
\]
for some non-zero $ P\in K[X_{0},X_{1},...,X_{r}]$. If $f$ is not $\partial$-algebraic,
it is called $\partial$-transcendental, or \emph{hypertranscendental}.
\begin{thm}[Main Theorem]
\label{thm:Main theorem} Assume that $f\in F$ satisfies a linear
homogeneous $\phi$-difference equation over $K$. If $f$ is $\partial$-algebraic
over $K$, then $f\in S$.
\end{thm}

In the rank one case, we can be more precise.
\begin{thm}
\label{thm:rank one main theorem}Assume that $f\in F$ satisfies
\[
\phi(f)=af
\]
with $a\in K.$ If $f$ is $\partial$-algebraic over $K$, then $z^{-r}f\in K$
for some $r\in\mathbb{Z}$.
\end{thm}

We shall deduce the rank one case from the following Proposition, whose proof is postponed to Section \ref{subsec:proofrankone}, and which holds for an arbitrary $(\phi,\delta)$-field $\mathcal{F}$ with algebraically closed field of $\phi$-constants, in lieu of $F$.
\begin{prop}
\label{prop: rank 1 proposition}Let $\mathcal{F}$ be a $(\phi,\delta)$-field
containing $K'=K(z)$ with $\phi\circ\delta=\delta\circ\phi$ and
$\mathcal{F}^{\phi}=\mathbb{C}.$ Let $f\in\mathcal{F}^{\times}$ satisfy
$\phi(f)=af$ with $a\in K^{\times}.$ Then $f$ is $\delta$-algebraic
over $K'$ if and only if
\[
a=c\frac{\phi(b)}{b}
\]
for some $c\in\mathbb{C}^{\times}$ and $b\in K^{\times}.$
\end{prop}

\begin{proof}
(That Proposition \ref{prop: rank 1 proposition} implies Theorem
\ref{thm:rank one main theorem}). First, note that
\begin{itemize}
\item $f$ is $\partial$-algebraic over $K$ if and only if
\item $\mathrm{tr.deg.}K(f,\partial f,\partial^{2}f,...)/K<\infty,$ if
and only if
\item $\mathrm{tr.deg.}K(z,f,\delta f,\delta^{2}f,...)/K'<\infty,$ if and
only if
\item $f$ is $\delta$-algebraic over $K'.$
\end{itemize}
Assume that we have proved the Proposition and $f$ is as in Theorem
\ref{thm:rank one main theorem}. Assume that $f$ is $\partial$-algebraic
over $K,$ hence $\delta$-algebraic over $K'.$ By the Proposition
$a=c\phi(b)/b$ for some $c\in\mathbb{C}^{\times}$ and $b\in K^{\times}.$
It follows that $g=f/b\in F$ satisfies
\[
\phi(g)=cg.
\]
Since $g$ is a power series in $z$ with constant coefficients, this
forces $g=Cz^{r}$ for $C\in\mathbb{C}$ and $r\in\mathbb{Z}$ (and
$c=q^{r}).$ Thus $z^{-r}f=Cb\in K$ as desired. Note that the fact
that $F=\mathbb{C}((z))$ enters the proof only at the last step.
\end{proof}

\subsection{Some properties of divisors}

For an abelian group $R$ let
\[
\mathcal{D}_{\Lambda}(R)=\mathrm{Div}(\mathbb{C}/\Lambda;R)
\]
be the group of $\Lambda$-periodic divisors with values in $R$.
We write $D\in\mathcal{D}_{\Lambda}(R)$ either as a finite linear
combination
\[
D=\sum_{\xi\in\mathbb{C}/\Lambda}r_{\xi}[\xi] 
\]
or as a $\Lambda$-periodic, discretely supported function $D:\mathbb{C}\to R,$
$D(\xi)=r_{\xi}.$ The support $\mathrm{supp}(D)$ is the set
of $\xi\in\mathbb{C}$ where $D(\xi)\ne0.$

We let $\mathcal{D}_{\Lambda}^{0}(R)$ be the subgroup of divisors
of degree 0, where
\[
\mathrm{deg}_{\Lambda}:\mathcal{D}_{\Lambda}(R)\to R,\,\,\,\mathrm{deg}_{\Lambda}(D)=\sum_{\xi\in\mathbb{C}/\Lambda}r_{\xi}.
\]
We let $\mathcal{P}_{\Lambda}\subset\mathcal{D}_{\Lambda}^{0}(\mathbb{Z})$
be the subgroup of principal divisors, \emph{i.e.}, of divisors of the shape
$\mathrm{div}(f)$ for $f\in K_{\Lambda}^{\times}.$ By the Abel-Jacobi
theorem a $\mathbb{Z}$-valued divisor $D=\sum_{\xi\in\mathbb{C}/\Lambda}r_{\xi}[\xi]$
is principal if and only if $\mathrm{deg}_{\Lambda}(D)=0$ and
\[
s_{\Lambda}(D)=\sum_{\xi\in\mathbb{C}/\Lambda}r_{\xi}\xi=0\in\mathbb{C}/\Lambda.
\]

We let $\phi\in End(\mathcal{D}_{\Lambda})$ be defined as $\phi(D)(\xi)=D(q\xi)$;
alternatively, if $D=\sum_{\xi\in\mathbb{C}/\Lambda}r_{\xi}[\xi],$
then
\[
\phi(D)=\sum_{\xi\in\mathbb{C}/\Lambda}r_{q\xi}[\xi].
\]
Note:
\begin{itemize}
\item $\mathrm{div}(\phi(f))=\phi(\mathrm{div}(f))$ for $f\in K_{\Lambda}^{\times}.$
\item $\mathrm{deg}_{\Lambda}(\phi(D))=q^{2}\mathrm{deg}_{\Lambda}(D).$
\item $qs_{\Lambda}(\phi(D))=q^{2}s_{\Lambda}(D).$
\end{itemize}
To prove the last point, let $D=\sum_{\xi\in\mathbb{C}/\Lambda}r_{\xi}[\xi]$.
Then
\[
qs_{\Lambda}(\phi(D))=q\sum_{\xi\in\mathbb{C}/\Lambda}r_{q\xi}\xi=\sum_{\xi\in\mathbb{C}/\Lambda}r_{q\xi}q\xi=q^{2}\sum_{\eta\in\mathbb{C}/\Lambda}r_{\eta}\eta=q^{2}s_{\Lambda}(D)
\]
because for every $\eta$ there are $q^{2}$ values of $\xi$ with
$q\xi=\eta.$
\begin{lem}
\label{lem:divisors lemma}(i) For any abelian group $R,$ $\phi-1\in End(\mathcal{D}_{\Lambda}(R))$
is injective.

(ii) If $D\in\mathcal{D}_{\Lambda}(\mathbb{R})$ and $(\phi-1)(D)\in\mathcal{D}_{\Lambda}(\mathbb{Z})$,
then $D\in\mathcal{D}_{\Lambda}(\mathbb{Z}).$

(iii) If $D\in\mathcal{D}_{\Lambda}(\mathbb{R})$ and $(\phi-1)(D)\in\mathcal{P}_{\Lambda},$
then $D\in\mathcal{P}_{\Lambda'}$ for $\Lambda'=q(q-1)\Lambda.$
\end{lem}

\begin{proof}
(i) If $D$ is $\Lambda$-periodic and non-zero, choose $0\ne\xi\in\mathrm{supp}(D).$
If $D=\phi(D)=\phi^{2}(D)=\cdots$ then $\xi,\xi/q,\xi/q^{2},...$
are all in the support of $D,$ contradicting the fact that the support
is discrete.

(ii) If the claim is false, the image of $D$ in $\mathcal{D}_{\Lambda}(\mathbb{R}/\mathbb{Z})$
is annihilated by $\phi-1$ and non-zero, contradicting (i).

(iii) Assume that $(\phi-1)(D)\in\mathcal{P}_{\Lambda}.$ By (ii),
$D\in\mathcal{D}_{\Lambda}(\mathbb{Z}).$ Since 
\[
0=\mathrm{deg}_{\Lambda}((\phi-1)(D))=(q^{2}-1)\mathrm{deg}_{\Lambda}(D)
\]
we must have $D\in\mathcal{D}_{\Lambda}^{0}(\mathbb{Z})$. Now
\[
q(q-1)s_{\Lambda}(D)=qs_{\Lambda}((\phi-1)(D))=0.
\]
Let $\Pi$ be a fundamental parallelogram for $\Lambda,$ and $m=q(q-1).$
If $D=\sum_{\xi\in\Pi}n_{\xi}[\xi{}_{\Lambda}]$ where $\xi{}_{\Lambda}=\xi\mod\Lambda$,
then
\[
s_{\Lambda'}(D)=\sum_{i=0}^{m-1}\sum_{j=0}^{m-1}\sum_{\xi\in\Pi}n_{\xi}(\xi+i\omega_{1}+j\omega_{2})\mod\Lambda'=m^{2}\sum_{\xi\in\Pi}n_{\xi}\xi\mod\Lambda',
\]
where $\omega_{1}$ and $\omega_{2}$ span $\Pi$ (recall $\sum n_{\xi}=0$).
But $m\sum_{\xi\in\Pi}n_{\xi}\xi\in\Lambda,$ so 
\[
m^{2}\sum_{\xi\in\Pi}n_{\xi}\xi\in m\Lambda=\Lambda'.
\]
\end{proof}

\subsection{Proof of Proposition \ref{prop: rank 1 proposition}}\label{subsec:proofrankone}

\subsubsection{First steps}

Assume that $f\in\mathcal{F}$ satisfies
\[
\phi(f)=af.
\]
If $a=c\phi(b)/b$ for $c\in\mathbb{C}^{\times}$ and $b\in K^{\times},$
then replacing $f$ by $f/b$ we may assume that $a\in\mathbb{C}^{\times}.$
Then
\[
\phi(\frac{\delta f}{f})=\frac{\delta(\phi(f))}{\phi(f)}=\frac{\delta f}{f},
\]
so $\delta f/f\in\mathcal{F}^{\phi}=\mathbb{C}$ and $f$ is $\delta$-algebraic
over $K'.$

Conversely, assume that $f$ is $\delta$-algebraic over $K'.$ Let
$u=\frac{\delta f}{f}\in\mathcal{F},$ so that
\[
\phi(u)=u+\frac{\delta a}{a}.
\]
Since $u$, like $f,$ is also $\delta$-algebraic over $K',$ Theorem
C.8 of \cite{D-H21}\footnote{See also the proof of Lemma \ref{lem:Key Lemma-1-1} below, where
the same argument appears again.} shows that there exists a monic operator
$
\mathscr{L}\in\mathbb{C}[\delta]
$
and a $v\in K'$ such that
\[
\mathscr{L}(\frac{\delta a}{a})=\phi(v)-v.
\]
Embed $K(z)$ in $K((z))$ via the completion at 0, where we regard
$K$ as the field of constants. Extend $\phi$ to $K((z))$ so that
$\phi(z)=qz.$ Then the embedding is $\phi$-compatible. (Warning:
$K((z))$ can not be regarded as a subfield of $F=\mathbb{C}((z))$,
despite the fact that $K\subset F.$) Write
\[
v=\sum_{i\ge s}v_{i}z^{i}\in K((z))
\]
with $v_{i}\in K.$ Since $a\in K,$ $\mathscr{L}(\frac{\delta a}{a})$
in fact lies in $K[z]\subset K(z)$ and as a polynomial in $z$ with
coefficients from $K$ it looks like
\[
\mathscr{L}(\frac{\delta a}{a})=\partial^{\ell}(\frac{\partial a}{a})z^{\ell+1}+\mathrm{lower\,\,terms.}
\]
Here $\ell=\mathrm{deg}(\mathscr{L}).$ Comparing the expansions we
see that
\begin{equation}
\partial^{\ell}(\frac{\partial a}{a})=q^{\ell+1}\phi(v_{\ell+1})-v_{\ell+1}.\label{eq:basic relation rank 1}
\end{equation}

\subsubsection{Completion of the proof}
We shall deduce Proposition \ref{prop: rank 1 proposition} from equation
$(\ref{eq:basic relation rank 1}).$ Quite generally, if $h\in K_{\Lambda}$
and $z_0\in\mathbb{C},$ and if
\[
h(z)=\sum_{n>>-\infty}a_{n}(h,z_0)(z-z_0)^{n}
\]
is the Laurent expansion of $h$ at $z_0$, we let for $\ell\ge1,$
\[
a_{-\ell}(h)=\sum_{z_0\in\mathbb{C}/\Lambda}a_{-\ell}(h,z_0)[z_0]\in\mathcal{D}_{\Lambda}(\mathbb{C}).
\]
We call it the degree $-\ell$ polar divisor of $h$. Note that $a_{n}(h,z_0)$
are $\Lambda$-periodic in $z_0$, so this divisor is well-defined. 
Since for $z$ near $z_0$ we have
\[
\phi h(z)=h(qz)=\sum_{n>>-\infty}a_{n}(h,qz_0)(qz-qz_0)^{n}=\sum_{n>>-\infty}q^{n}a_{n}(h,qz_0)(z-z_0)^{n},
\]
we get
\[
a_{n}(\phi h,z_0)=q^{n}a_{n}(h,qz_0),
\]
hence
\[
a_{-\ell}(\phi h)=q^{-\ell}\sum_{z_0 \in\mathbb{C}/\Lambda}a_{-\ell}(h,qz_0)[z_0]=q^{-\ell}\phi(a_{-\ell}(h)).
\]

Now, $(\ref{eq:basic relation rank 1})$ yields
\[
a_{-\ell-1}(\partial^{\ell}(\frac{\partial a}{a}))=q^{\ell+1}a_{-\ell-1}(\phi(v_{\ell+1}))-a_{-\ell-1}(v_{\ell+1})=(\phi-1)(a_{-\ell-1}(v_{\ell+1})).
\]
On the other hand
\[
a_{-\ell-1}(\partial^{\ell}(\frac{\partial a}{a}))=(-1)^{\ell}\ell!a_{-1}(\frac{\partial a}{a})=(-1)^{\ell}\ell!\mathrm{div}(a).
\]
It follows that
\[
\mathrm{div}(a)=(\phi-1)(D)
\]
where $D=(-1)^{\ell}a_{-\ell-1}(v_{\ell+1})/\ell!$. By part (iii)
of Lemma \ref{lem:divisors lemma} we deduce that there exists $b\in K_{\Lambda'}^{\times}$,
$\Lambda'=q(q-1)\Lambda,$ such that $D=\mathrm{div}(b).$ In particular
\[
\mathrm{div}(a)=\mathrm{div}\frac{\phi(b)}{b},
\]
so
\[
a=c\frac{\phi(b)}{b}
\]
for some $c\in\mathbb{C}^{\times},$ as desired.

\section{Simple Galois group}

\subsection{Comparison of the Galois groups over $K$ and $K(z)$\label{subsec:Comparison of Galois groups}}

Consider the system $(\ref{eq:general diference system}),$ with $A\in \GL_{n}(K)$, and its Galois group $G$.
We want to compare $G$ with the Galois group $H$ of the same system over $K'$.
 
 In order to compare $H$ and $G$, we adapt the proof of \cite[Proposition 1.6]{DHR} which was written for Mahler difference systems. Let $R'$ be a Picard-Vessiot ring for $\phi(y)=Ay$ over $K'$ and let $L'$ be the associated PV-extension.  For $U \in \GL_n(R')$  a fundamental solution matrix, it is easily seen that $R=K[U,\det(U)^{-1}]$ is a PV-ring for $\phi(y)=Ay$ over $K$ and we denote by $L$ the associated PV-extension. The inclusion $R \subset R'$ yields a closed immersion $\iota: H \hookrightarrow G$. Identifying $H$ with its image in $G$, we can consider $L^H ={L'}^H \cap L=K' \cap L$. Since $L^H/K$ is a sub-$\phi$-extension of $K'/K$, it is easily seen that there exists an integer $N$ such that $L^H=K(z^N)$. Now the Galois correspondence yields that $H$ is normal in $G$ and that $G/H$ is isomorphic to $\mathrm{Aut}_{\phi}(L^H/K)$ which is trivial if $N$ equals zero and $\mathbb{G}_m$ otherwise. 
 
\begin{lem}
\label{lem:going up and down}Assume (by replacing $\phi$ by some
$\phi^{r}$) that $L$ is a field. Then:
\begin{enumerate}[i)]

\item  $G$ is connected.

\item  If $G$ is simple\footnote{By ``$G$ is simple'', we mean that it is noncommutative and has no proper nontrivial normal closed subgroup.}, then $H=G$.

\item  $G$ is solvable if and only if $H$ is solvable.
\end{enumerate}

\end{lem}

\begin{proof}
\begin{enumerate}[i)]
\item  If $G^{0}$ is the connected component of $G,$ then it is normal,
and the finite group $G/G^{0}$ is, by the Galois correspondence,
the Galois group of a finite extension of $K$ contained in $L$ to
which $\phi$ extends as an automorphism. But such an extension must
be $K$ itself by Lemma $\ref{lem:no finite phi extensions of K}$,
so $G=G^{0}.$

\item If the normal subgroup $H$ is trivial, $G$ would be $\mathbb{G}_{m}$, so by the simplicity of $G$ we must have
 $H=G$.

\item This follows
from the fact that $G/H$ is either trivial or $\mathbb{G}_{m}$.
\end{enumerate}

\end{proof}

Note that  the group $H$ need not be connected. For example,
let $(M,N)=1,$ $y=x^{M}$, $z=x^{N}$ and $q=q_{1}^{N}.$ If $A=(q_{1}^{M})$
and $A'=(q_{1}^{M})\oplus(q_{1}^{N})$ then $L=K(y)$ with $\phi(y)=q_{1}^{M}y,$
$K'=K(z)$ with $\phi(z)=q_{1}^{N}x$ and $L'=K(x)$ with $\phi(x)=q_{1}x.$
We would then have $H\simeq\mu_{N}\subset G=\mathbb{G}_{m}.$

\subsection{The $G$ simple case}

Consider now the $\delta$-parametrized PV extension
\[
\mathcal{L}'=K'\left\langle U\right\rangle _{\delta}
\]
and the $\delta$-parametrized Galois group scheme $\mathcal{H}$ of $\mathcal{L}'$ over $K'$. 

As explained before, $\mathcal{H}$ is a Zariski-dense differential subgroup scheme of $H=\mathrm{Aut}_{\phi}(L'/K')$, in the sense that its
points $\mathcal{H}(\widetilde{C})$ in a $\delta$-closure $\widetilde{C}$
of $\mathbb{C}$ form a Zariski dense subset of $H(\widetilde{C}).$
We continue to assume that $L,L'$ and $\mathcal{L}'$ are fields.

\begin{prop}
\label{prop:Main in simple case}Under the assumption that $G=\mathrm{Aut}_{\phi}(L/K)$
is simple, we have 
\[
\mathcal{H}=[\delta]_{*}H.
\]
Consequently, $\mathrm{tr.deg.}(L/K)=\mathrm{tr.deg.}(L'/K')=\delta\mathrm{tr.deg.}(\mathcal{L}'/K').$
\end{prop}

This Proposition seems identical to Lemma 5.1 of \cite{Arr-S} and
similar to Proposition 4.11 of \cite{A-D-H-W}. However, because of
the need to descend, first from a differentially closed field of $\phi$-constants
to $\mathbb{C},$ and then from $K'$ to $K,$ we are forced to go
through some of the arguments with care.
\begin{proof}
As we have seen in Lemma \ref{lem:going up and down}, when $G$ is
simple, so is $H$, and $H\simeq G.$ 

Let $\widetilde{K'}$ be a $(\phi,\delta)$-extension of $K'$ whose
field of $\phi$-constants, $\widetilde{C}=(\widetilde{K'})^{\phi}$
is a differential closure of $\mathbb{C}$. It is known that $\widetilde{C}^{\delta}$
is just $\mathbb{C}$ (\cite{M-MTDF}, Lemma 2.11). Let $\widetilde{L'}$
be the PV extension over $\widetilde{K'}$, and $\widetilde{\mathcal{L}'}$
the $\delta$-parametrized PV extension. See \S~\ref{subsec:The--parametrized-Galois}
for the construction of $\widetilde{K'},\widetilde{L'}$ and $\widetilde{\mathcal{L}'}$,
starting from a differential closure $\widetilde{C}$ of $\mathbb{C},$
and their relation to $K',L'$ and $\mathcal{L}'$.

Let $\widetilde{H}$ and $\mathcal{\widetilde{H}}$ be the difference
Galois group and the $\delta$-parametrized difference Galois group
over $\widetilde{K}'$. We regard them as subgroups of $\GL_{n}(\widetilde{C}),$
having fixed a fundamental matrix $U$. Since $\widetilde{H}=H(\widetilde{C})$
and $\widetilde{\mathcal{H}}=\mathcal{H}(\widetilde{C})$ (see \S~\ref{subsec:The--parametrized-Galois})
we only need to show that $\widetilde{\mathcal{H}}=[\delta]_*\widetilde{H}.$

Since the field of $\phi$-constants $\widetilde{C}$ is differentially
closed and $\widetilde{H}$ is simple, Theorem \ref{thm: Cassidy's theorem on simple Zariski closure}
implies that if  $\widetilde{\mathcal{H}} \subsetneq [\delta]_*\widetilde{H}$, there exists a matrix
$E\in \GL_{n}(\widetilde{C})$ such that
\[
\widetilde{\mathcal{H}}\subset E\cdot \GL_{n}(\mathbb{C})\cdot E^{-1}.
\]

Letting $D=\delta(E)E^{-1}\in\mathfrak{gl}_{n}(\widetilde{C})$ this
implies that for any $\sigma\in\mathcal{\widetilde{H}},$
\[
\delta(V_{\sigma})=DV_{\sigma}-V_{\sigma}D.
\]
As in Proposition \ref{prop:Galois criterion for integrability-1}
this implies that the system $\phi(y)=Ay$ is $\delta$-integrable
over $\widetilde{K}'$. In fact, the matrix
\[
B=\delta(U)U^{-1}-UDU^{-1}
\]
which gives the $\delta$-integrability relation
\begin{equation}
\delta(A)=\phi(B)A-AB\label{eq:delta integrability again}
\end{equation}
lies in $\mathfrak{gl}_{n}(\widetilde{C}\otimes_{\mathbb{C}}\mathcal{R}')$
by construction, as $U,U^{-1}$ and $\delta(U)$ have entries in $\mathcal{R}'\subset\mathcal{L}'$
and $D$ has entries in $\widetilde{C}.$ It is shown, in Proposition
\ref{prop:Galois criterion for integrability-1}, to be invariant
under every 
\[
\sigma\in \widetilde{\mathcal{H}}=\mathcal{H}(\widetilde{C})=\mathrm{Aut}_{\phi,\delta}(\widetilde{C}\otimes_{\mathbb{C}}\mathcal{R}'/\widetilde{C}\otimes_{\mathbb{C}}K').
\]
By a refinement of the parametrized Galois correspondence, the matrix $B$ therefore lies in $\mathfrak{gl}_{n}(\widetilde{C}\otimes_{\mathbb{C}}K')$
 and not just in $\mathfrak{gl}_{n}(\widetilde{K}')$ (see Proposition~\ref{prop:refinmentparamgalcorrespondence} in Appendix~\ref{subsec:Galois-invariants}).  
Now let $\{\omega_{\alpha}\}$
be a basis of $\widetilde{C}$ as a vector space over $\mathbb{C}$
and write $B=\sum_{\alpha}\omega_{\alpha}\otimes B_{\alpha}$ with
$B_{\alpha}\in\mathfrak{gl}_{n}(K')$. We may assume that $\omega_{\alpha_{0}}=1.$
Since $\phi(\omega_{\alpha})=\omega_{\alpha},$ and since $A$ has
entries in $K',$ decomposing the $\delta$-integrability relation
to its $\alpha$-components we get
\[
\delta(A)=\phi(B_{\alpha_{0}})A-AB_{\alpha_{0}}.
\]
Thus without loss of generality we may assume that $B\in\mathfrak{gl}_{n}(K').$
In other words, we have descended the integrability of the system
over $\widetilde{K}'$ to the integrability of the same system over
$K'$.

Equivalently, $(\ref{eq:delta integrability again})$ may be written
\begin{equation}
\partial A=q\phi(z^{-1}B)A-A(z^{-1}B).\label{eq:partialAequation}
\end{equation}
Consider the field $K((z)),$ to which $\phi$ is extended as an automorphism
so that $\phi(z)=qz.$ Embedding $K(z)$ in $K((z))$ via the completion
at 0, we expand
\[
z^{-1}B=\sum_{-\infty<<\ell}B_{\ell}z^{\ell},\,\,\,\phi(z^{-1}B)=\sum_{-\infty<<\ell}q^{\ell}\phi(B_{\ell})z^{\ell},
\]
where $B_{\ell}\in\mathfrak{gl}_{n}(K).$ Substituting into $(\ref{eq:partialAequation})$
and comparing the $\ell=0$ terms we get (as $A\in \GL_{n}(K))$
\[
\partial A=q\phi(B_{0})A-AB_{0}.
\]
It follows that the system $(\ref{eq:general diference system})$,
proven to be $\delta$-integrable over $K',$ is also $\partial$-integrable
over $K$.

We may now apply Corollary \ref{cor: integrability implies solvability-1}
to conclude that $G$ is solvable. This contradicts the assumption
that $G$ is simple.

The assertion on the transcendence degrees follows from
\[
\mathrm{tr.deg.}(L/K)=\dim G,\,\,\,\mathrm{tr.deg.}(L'/K')=\dim H,\,\,\,\delta\mathrm{tr.deg.}(\mathcal{L}'/K')=\delta\dim\mathcal{H},
\]
(the first two equalities are well-known consequences of the \emph{torsor
theorem}; for the last one, a consequence of the same theorem in the
$\delta$-parametrized setup, see \cite{H-S08}, Proposition 6.26),
and $
\delta\dim[\delta]_{*}H=\dim H$ by \cite[Prop. 10 p.200]{Kol}.
\end{proof}

\section{The Main Theorem in the irreducible case}

\subsection{Reduction steps}

To prove Theorem \ref{thm:Main theorem} it is enough to prove the
following theorem.
\begin{thm}
\label{thm:Main thm, system version-1} Let $A\in \GL_{n}(K)$ and
consider the system $(\ref{eq:general diference system}).$ Let $u=\,^{t}(u_{1},...,u_{n})\in F^{n}$
be a solution of
\[
\phi(u)=Au,
\]
such that the $u_{i}$ are $\partial$-algebraic over $K$. Then every
$u_{i}\in S=K[z,z^{-1},\zeta(z,\Lambda)].$
\end{thm}

Indeed, suppose this is proven, and $f\in F$ is a solution of the
linear homogenous $\phi$-difference equation 
\[
\phi^{n}(y)+a_{1}\phi^{n-1}(y)+\cdots+a_{n-1}\phi(y)+a_{n}y=0
\]
($a_{i}\in K$). We may assume that $a_{n}\ne0,$ otherwise $f$ satisfies
a similar equation of order $n-1.$ The \emph{companion matrix}
\[
A=\left(\begin{array}{ccccc}
0 & 1\\
 & 0 & 1\\
 &  &  & \ddots\\
 &  &  & 0 & 1\\
-a_{n} & -a_{n-1} & \cdots & -a_{2} & -a_{1}
\end{array}\right),
\]
therefore lies in $\GL_{n}(K)$, and
\[
u=\,^{t}(f,\phi(f),...,\phi^{n-1}(f))
\]
is a solution of $\phi(u)=Au.$ If $f$ is $\partial$-algebraic over
$K,$ then so is every $\phi^{i}(f),$ and we may conclude that the
entries of $u,$ and in particular $f,$ lie in $S.$

We recall that the $u_{i}$ are $\partial$-algebraic over $K$ if
and only if they are $\delta$-algebraic over $K'=K(z).$
\begin{lem}[Reduction Lemma]
\label{lem:Reduction Lemma-1} Without loss of generality, we may
assume in Theorem \ref{thm:Main thm, system version-1}:
\begin{enumerate}
\item The PV extension $L$ (respectively $L'$) of \eqref{eq:general diference system} over $K$ (sesp. $K'$), and the $\delta$-parametrized PV extension $\mathcal{L}'$ are
fields and we have $L \subset L' \subset \mathcal{L}'$.
\item The field $K'\left\langle u\right\rangle _{\delta}=K(z,u_{i},\delta(u_{i}),\delta^{2}(u_{i}),...)\subset F$
is embedded as a subfield of $\mathcal{L}'$.
\item The difference Galois group $G=\mathrm{Aut}_{\phi}(L/K)$ is connected.
\end{enumerate}
\end{lem}

\begin{proof}

A-priori, a $PPV$ extension $\mathcal{L}'$ has the form
\[
\mathcal{L}'=\mathcal{L}_{1}'\times\cdots\times\mathcal{L}_{r}'
\]
where the $\mathcal{L}_{i}'$ are fields, $\phi(\mathcal{L}_{i}')=\mathcal{L}'_{i+1\mod r}$
and $\delta(\mathcal{L}')\subset\mathcal{L}_{i}'$. Replacing $\phi$
by $\phi^{r}$ and $A$ by $A^{(r)}=\phi^{r-1}(A) \dots \phi(A)A$ , keeping $u$ unchanged, $\mathcal{L}'$
gets replaced by $\mathcal{L}_{1}'$(see \S~\ref{subsubsec:PPVext}). We may therefore assume that
$\mathcal{L}'$  is a field, and as $L'\subset\mathcal{L}'$, the ordinary PV extension $L'$ is a field too. Moreover, for $U  \in \GL_n(L')$ a fundamental solution matrix, the subfield $L=K(U)$ of $L'$ satisfies $L^{\phi} =L'^{\phi}=\mathbb{C}$ and is thereby a PV extension  for \eqref{eq:general diference system}  over $K$.  This settles (1). Point (2), the
embedding of $K'(u)$ in $L',$ and of $K'\left\langle u\right\rangle _{\delta}$
in $\mathcal{L}'$, is proven in \cite{A-D-H}, Proposition 4.10. See
also Remark \ref{rem:PV  ext containing a given solution}. The vector
$u$ becomes then a linear combination, with constant coefficients,
of the columns of the fundamental matrix $U.$ Point (3) has already
been observed before, as a consequence of the fact that $K$ does
not admit non-trivial finite $\phi$-field extensions.
\end{proof}

\subsection{The Main Theorem in the irreducible case}
\begin{prop}
\label{prop:Main Theorem irreducible case}Assume that the $\phi$-module
$W$ associated with $(\ref{eq:general diference system})$ over $K$ is irreducible.
If $\mathcal{U}_{a}\ne0,$ \emph{i.e.}, if there exists a non-zero solution
$u$ all of whose coordinates are $\partial$-algebraic over $K$,
then the rank $n$ must be equal to $1$.

Furthermore, if the given $\partial$-algebraic solution $u$ lies
in $F^{n},$ then Theorem \ref{thm:Main thm, system version-1} holds
true.
\end{prop}

\begin{proof}
Assume that $W$ is irreducible. By Lemma \ref{lem:Irreducibility lemma},
$W_{K'}$ is an irreducible $(K',\phi)$-module. By the Tannakian
correspondence, Proposition \ref{prop:Tannakian correspondence},
the solution space
\[
\mathcal{U}=U\mathbb{C}^{n}=W_{L'}^{\Phi}=W_\mathcal{L}'^{\Phi}\subset W_{\mathcal{L}'}=\mathcal{L}'^{n}
\]
is an irreducible representation of $H=\mathrm{Aut}_{\phi}(L'/K')$.

By Corollary \ref{cor:Galois invariance of delta algebraic solutions}, the vector space $\mathcal{U}_{a}$ of differentially algebraic  over $K$ solutions
 is $H$-invariant. If $\mathcal{U}_{a}\ne0,$ we
must have $\mathcal{U}=\mathcal{U}_{a},$ all the entries of $U$
are $\delta$-algebraic, and the field ${\mathcal{L}'}_{a}$  formed by the differentially algebraic over $K$ elements coincides with $\mathcal{L}'.$

We claim that this forces $H$ to be solvable, contradicting the irreducibility
of $\mathcal{U}$, unless $n=1.$

Suppose, therefore, that $H$ is not solvable. Then $G$ is not solvable
either by Lemma~\ref{lem:going up and down} and as it is connected, it must have a simple quotient $G_{1}$
with $0<\dim(G_{1})$ (indeed, the quotient of $G$ by its radical is non-trivial, connected and semi-simple, so admits a simple connected quotient). By the Galois correspondence theorem this quotient
is the difference Galois group of a normal subextension $M_{1}\subset L,$
a PV extension of another system $\phi(y)=A_{1}y$ over $K$. Since $M_1 \subset L \subset L' \subset \mathcal{L}'$,  the 
$\delta$-parametrized PV extension of $\phi(y)=A_{1}y$ over $K'$ is a subfield $\mathcal{M}'_{1}\subset\mathcal{L}'.$
It follows from the equality $\mathcal{L}_{a}'=\mathcal{L}'$ that
\[
\delta\mathrm{tr.deg.}(\mathcal{M}'_{1}/K')=0.
\]
Proposition \ref{prop:Main in simple case}, applied to the system
$\phi(y)=A_{1}y$ and the simple Galois group $G_{1},$ yields the
contradiction
\[
0<\dim(G_{1})=\mathrm{tr.deg.}(M_{1}/K)=\delta\mathrm{tr.deg.}(\mathcal{M}'_{1}/K')=0.
\]
This shows that $H$ must be solvable, and concludes the proof that
$n=1.$

Finally, if in addition $u\in F,$ by Theorem \ref{thm:rank one main theorem}
it belongs to $S$.
\end{proof}

\section{Conclusion of the proof}

\subsection{Reduction to the case of an inhomogeneous rank 1 equation}

Keeping the notation as above, we can now complete the proof of Theorem
\ref{thm:Main thm, system version-1}, and with it the proof of Theorem
\ref{thm:Main theorem}.

If $W$ is irreducible, we have just proved the theorem. Assume that
it is reducible, and let $W_{1}\subset W$ be an\emph{ irreducible}
$\phi$-submodule of dimension $1\le n_{1}<n.$ In an appropriate
basis, our system of equations and the solution $u$ look like
\[
\phi\left(\begin{array}{c}
u'\\
u''
\end{array}\right)=\left(\begin{array}{cc}
A_{1} & A_{12}\\
0 & A_{2}
\end{array}\right)\left(\begin{array}{c}
u'\\
u''
\end{array}\right),
\]
where $A_{1}\in \GL_{n_{1}}(K),$ $u'$ is a vector of length $n_{1}$
and $u''$ is a vector of length $n-n_{1}.$ By induction on $n,$
and since the coordinates of $u''$ are all $\partial$-algebraic
over $K$ and lie in $F,$ we deduce that $u''\in S^{n-n_{1}}.$

As in \cite{A-D-H}, \S5.3 (see also \cite{A-D-H-W}, \S4.1), we are
going to show that $n_{1}=1.$ Let 
\[
E=Frac(S)=K(z,\zeta(z,\Lambda))\subset F
\]
be the field of fractions of $S.$ The equation $(\ref{eq:general diference system})$
shows that the coordinates of $u$ span a finite dimensional $\phi$-invariant
$K$-subspace of $F$. If they belong to $E,$ then by Corollary \ref{cor: Crterion for elts of E to be in S},
they must lie in $S,$ and we are done.

Assume therefore that one of the coordinates of $u'$ does not belong
to $E$. Consider the system of equations $(\ref{eq:general diference system})$
over $E,$ and let $L_{E}=E(U)\subset\mathcal{L}_{E}$ be its PV and
PPV extensions over $E$. As in Lemma~\ref{lem:Reduction Lemma-1},
we can consider an iterate of the system $\phi(y)=Ay$ in order to  assume that $L_{E}$ is a field and $E(u')$ is a subfield of $L_{E}$. Since the $\phi$-module $W_1$ is irreducible, its Galois group $G_1$ is irreducible. It is also connected as a quotient of the connected group $\mathrm{Aut}_{\phi}(L/K)$. Thus, for any positive integer $r$,  the $\phi^r$-module $(W_1,\Phi^r)$ is still irreducible. Indeed, its Galois group is $G_1$ by \cite[Cor. 1.17]{S-vdP97} and it acts irreducibly on the space of solution of $\phi^r(y)=A_1^{(r)}y$, which coincides with the space of solution of  $\phi(y)=A_1y$. The Tannakian equivalence yields the irreducibility of $(W_1,\Phi^r)$ as $\phi^r$-module for any $r$. Finally, one can easily consider the PV extension $L'$ over $K'$ as a subfield of  $L_{E}$ and   the PPV extension $\mathcal{L}'$ as  a subfield of $\mathcal{L}_E$. Since
$u'\notin E^{n_{1}},$ there exists a
\[
\tau\in\mathrm{Aut}_{\phi}(L_{E}/E)\subset\mathrm{Aut}_{\phi}(L'/K')=H
\]
with $v=\tau(u')-u'\ne0.$ Furthermore, by Corollary \ref{cor:Galois invariance of delta algebraic solutions},
and the assumption that the coordinates of $u$ are $\delta$-algebraic
over $K',$ the coordinates of $v$ are $\delta$-algebraic over $K'$
as well. By Lemma \ref{lem:Irreducibility lemma}, the $\phi$-module  $W_{1,K'}$ is
irreducible over $K'$. As $\tau$ fixes $A_{12}u'',$ our $v$ satisfies
\[
\phi(v)=A_{1}v,
\]
so the system corresponding to $W_{1,K'}$ admits a non-zero $\delta$-algebraic
solution. It follows now from Proposition \ref{prop:Main Theorem irreducible case}
that $n_{1}=1$.

\subsection{Rank 1 inhomogeneous equations}

\subsubsection{Reduction to the case $a\in\mathbb{C}^{\times}$}

We have arrived at the equation
\begin{equation}\label{eq:inhomogrankone}
\phi(w)=aw+b
\end{equation}
where $a=a_{11}\in K,$ $b=a_{12}u_{2}+\cdots+a_{1n}u_{n}\in S.$
We assume that $w=u_{1}\in F$ satisfies it, and is $\partial$-algebraic
over $K.$ To conclude the proof of the Main Theorem we must show
that $w\in S.$ By Corollary \ref{cor: Crterion for elts of E to be in S},
it is enough to show that $w\in E.$

We continue to work over $E$ as a ground field. Let $L_E$ and $\mathcal{L}_E$ be respectively the PV and the PPV extensions of \eqref{eq:inhomogrankone} over $E$ and let us consider  the associated  difference Galois
group $G_{E}=\mathrm{Aut}_{\phi}(L_{E}/E)$ and the $\delta$-parametrized Galois group scheme $\mathcal{G}_{E}$ of ${\mathcal{L}}_{E}$ over $E$.
If $w\in E,$ we are done. Thus we can assume that  this is not the case.
 Our first goal is
to show that we may assume $a\in\mathbb{C}^{\times}.$ Since $w \notin E$, 
by the Galois correspondence theorem there exists a $\tau\in G_{E}$ such that $\tau(w)\ne w$. Letting
$v=\tau(w)-w$ we arrive at
\[
\phi(v)=av.
\]
Since $w$ is $\delta$-algebraic over $K',$ Corollary \ref{cor:Galois invariance of delta algebraic solutions}
shows that $\tau(w)$ is $\delta$-algebraic as well, hence $v$ is
$\delta$-algebraic over $K'$. Proposition \ref{prop: rank 1 proposition}
(with $\mathcal{F}=\mathcal{L}_{E}$) implies that
\[
a=c\frac{\phi(\beta)}{\beta}
\]
for some $c\in\mathbb{C}^{\times}$ and $\beta\in K^{\times}.$ The
original equation becomes equivalent to
\[
\phi(\frac{w}{\beta})=c\frac{w}{\beta}+\frac{b}{\phi(\beta)}
\]
and $b/\phi(\beta)$ still lies in $S$. If we show $w/\beta\in S,$
then $w\in S$ as desired. This allows us to assume, from the beginning,
that $a\in\mathbb{C}^{\times}.$

\subsubsection{The proof when $a\in\mathbb{C}^{\times}$}

Recall that $v\in L_{E}$ is the fundamental matrix of the rank-1
equation $\phi(v)=av.$ Contrary to 
\[
E_{2}:=E(w)\subset F,
\]
and unless $a$ is a power of $q$, the field $E_{1}:=E(v)\subset L_{E}$
can not be embedded $\phi$-equivariantly in the Laurent power series
field $F$. A-priori, we only know that $\delta(v)\in \mathcal{L}_{E}$.
\begin{lem}
We have $\delta(v)=cv$ for some $c\in\mathbb{C}.$ Thus $E_{1}$
is a $(\phi,\delta)$-field.
\end{lem}

\begin{proof}
If $\delta(v)=0$ this is clear. Otherwise, one has
\[
\frac{\phi(\delta v)}{\delta v}=\frac{\delta(\phi v)}{\delta v}=\frac{\delta(av)}{\delta v}=a
\]
since $a\in\mathbb{C}.$ But $\frac{\phi(v)}{v} = a$ as well, so $\delta(v)/v$,
being fixed by $\phi,$ must belong to $\mathcal{L}_{E}^{\phi}=\mathbb{C}.$
\end{proof}
Consider $U'=\left(\begin{array}{cc}
v & w\\
0 & 1
\end{array}\right)\in \GL_{2}(L_{E})$ and $A'=\left(\begin{array}{cc}
a & b\\
0 & 1
\end{array}\right)\in \GL_{2}(E)$, so that
\[
\phi(U')=A'U',
\]
and $E(v,w)=E_{1}E_{2}\subset L_{E}$ is the PV extension of the linear
system $\phi(y)=A'y$ over $E.$ Let $\mathcal{E}=E\left\langle v,w\right\rangle _{\delta}\subset \mathcal{L}_{E}$
be the $\delta$-parametrized PV extension of this last system over
$E$.

The proof of the following Key Lemma resembles the proof of \cite{H-S08},
Proposition 3.8. See also the proof, in the case of two difference
operators, in \cite{dS23}, \S5.4.
\begin{lem}
\label{lem:Key Lemma-1-1}After possibly replacing $w$ by $z^{r}w,$
$a$ by $q^{r}a$ and $b$ by $q^{r}z^{r}b$ for some positive integer $r,$ there
exists a monic operator $\mathscr{L}\in\mathbb{C}[\delta]$, and a
function $f\in S_{0}$, such that
\begin{equation}
\mathscr{L}(b)=(\phi-a)(f).\label{eq:key equation-1-1}
\end{equation}
\end{lem}

\begin{rem*}
We do not rule out $f=\mathscr{L}(b)=0.$ In fact, this may well be
the solution if $b$ is annihilated by some operator from $\mathbb{C}[\delta],$
a condition which is easily verified to hold if and only if $b$ is
a Laurent polynomial in $z$.
\end{rem*}
\begin{proof}
The $\delta$-parametrized Galois group scheme $\mathcal{G}'$ of the system $\phi(y)=A'y$
over $E$ is a linear differential subgroup scheme
$
\left\{ \left(\begin{array}{cc}
\alpha & \beta\\
0 & 1
\end{array}\right) \in \GL_2 \right\} ,
$ defined over $\mathbb{C}$. 
Its intersection with the unipotent radical $\left\{ \left(\begin{array}{cc}
1 & \beta\\
0 & 1
\end{array}\right) \in \GL_2 \right\}$, denoted
$\mathcal{G}'_{u},$ is the $\delta$-parametrized\index{-parametrized}
Galois group scheme  of $\mathcal{L}_E$ over the $(\phi,\delta)$-field
\[
E_{1}=E(v).
\]
This $\mathcal{G}'_{u}$ is a linear differential subgroup scheme of the
additive group $[\delta]_{*}\mathbb{G}_{a}$ over $\mathbb{C}$. Since $w$ is $\delta$-algebraic, the torsor theorem for the parametrized Galois correspondence yields that 
$ \delta \rm{tr.deg.} (\mathcal{E} /E_1)=0= \delta\dim_{\mathbb{C}}(\mathcal{G}'_{u}),$ so, by \cite{Cas72},
Proposition 11, there must be a non-trivial linear operator $\mathscr{L}_{1}\in\mathbb{C}[\delta]$
such that
\[
\mathscr{L}_{1}(\beta_{\tau})=0
\]
for every $\tau=\left(\begin{array}{cc}
1 & \beta_{\tau}\\
0 & 1
\end{array}\right)\in\mathcal{G}'_{u}(C),$ for any $\delta$-ring extension $C$ of $\mathbb{C}$.

As
\[
\phi(\frac{w}{v})=\frac{w}{v}+\frac{b}{av}
\]
and $b/av\in E_{1},$ $\mathcal{E}$ is a PPV extension for the equation
\[
\phi(y)=y+b/av
\]
over $E_{1}.$ The action of $\tau\in\mathcal{G}'_{u}(C)$ is given
by
\[
\tau(\frac{w}{v})=\frac{w}{v}+\beta_{\tau}.
\]
It follows that for any $\tau\in\mathcal{G}'_{u}(C)$ we have, in
the base-changed PPV ring 
\[
E_{1}[w,\delta(w),...,w^{-1}]\otimes_{\mathbb{C}}C
\]
the equation

\[
\tau(\mathscr{L}_{1}(\frac{w}{v}))=\mathscr{L}_{1}(\frac{w}{v}+\beta_{\tau})=\mathscr{L}_{1}(\frac{w}{v}).
\]
By the Galois correspondence in the $\delta$-parametrized framework,
$\mathscr{L}_{1}(\frac{w}{v})\in E_{1}.$ Leibnitz' formula and the
equation $\delta(v)=cv$ imply
\[
v^{-1}\delta^{k}(w)=\sum_{j=0}^{k}\binom{k}{j}c^{k-j}\delta^{j}(\frac{w}{v}).
\]
Since a linear combination, with constant coefficients, of the $\delta^{j}(\frac{w}{v}),$
lies in $E_{1},$ so does a linear combination, with constant coefficients,
of $v^{-1}\delta^{k}(w)$. As $v\in E_{1},$ we find that for some
non-trivial $\mathscr{L}\in\mathbb{C}[\delta]$,
\[
f=\mathscr{L}(w)\in E_{1}.
\]
Since $\mathscr{L}$ commutes with $\phi,$ this $f$ satisfies $(\ref{eq:key equation-1-1}),$
but may not be in $E.$

By Lemma 4.7 of \cite{A-D-H-W} (with $E$ in the role of $L$ \emph{and}
$K,$ $v$ in the role of $x,$ and $L_{E}$ in the role of $L_{A}$)
we conclude that there exists an $f\in E$ (possibly different than
$\mathscr{L}(w)$) satisfying $(\ref{eq:key equation-1-1}).$ By Corollary
\ref{cor: Crterion for elts of E to be in S}, $f\in S.$

We have arrived at the two equations
\begin{equation}
\begin{cases}
\begin{array}{c}
(\phi-a)(w)=b\\
(\phi-a)(f)=\mathscr{L}(b).
\end{array}\end{cases}\label{eq:pair of equations-1-1}
\end{equation}
Let $Z$ be the operator of multiplication by $z.$ We have the relations
\[
Z\circ(\phi-a)=q^{-1}(\phi-qa)\circ Z,\,\,\,Z\circ\mathscr{L}(\delta)=\mathscr{L}(\delta-1)\circ Z.
\]
Multiplying the two equations by $z^{r},$ replacing $a$ by $q^{r}a,$
$b$ by $q^{r}z^{r}b$, $\mathscr{L}(\delta)$ by $\mathscr{L}(\delta-r),$
$w$ by $z^{r}w$ and $f$ by $z^{r}f,$ we get a similar pair of
equations, but we may assume now that $f\in S_{0},$ not only in $S$.
\end{proof}
Corollary \ref{cor:technical elliptic} implies now that for some
$h\in S$ and $d\in\mathbb{C}$ we have
\[
(\phi-a)(w-h)=dz^{r},
\]
with $d=0$ unless $a=q^{r}.$ As $w-h\in F,$ it is easily verified
that $w-h$ must be of the form $ez^{m}$ (with $e\in\mathbb{C}$
and $e=0$ unless $a=q^{m}$). It follows that $w\in S$, as desired.

\section{Appendix A: Proof of Proposition \ref{prop:technical_elliptic}}

We recall the proposition whose proof we give here, with a slight
change in notation.
\begin{prop*}
Let $f,g\in S_{\Lambda}=K_{\Lambda}[z,\zeta(z,\Lambda)],$ $a,c\in\mathbb{C}$
and $p\in\mathbb{C}[z]$ be such that
\begin{equation}
(\delta-c)(g)=(\phi-a)(f)+p.\label{eq:delta-phi}
\end{equation}
Then $g=(\phi-a)(u)+\widetilde{p}$ for some $u\in S_{\Lambda}$ and
$\widetilde{p}\in\mathbb{C}[z].$ Furthermore, if $a=q^{r}$ for some
$r\ge0$ we can take $\widetilde{p}=dz^{r},$ $d\in\mathbb{C},$ and
otherwise we can take $\widetilde{p}=0.$
\end{prop*}
We start with some lemmas. To ease notation, write from now on $\zeta:=\zeta(z,\Lambda)$.
Note that
\[
\phi(\zeta)=q\zeta+f_{\zeta}
\]
where $f_{\zeta}=\zeta(qz,\Lambda)-q\zeta(z,\Lambda)\in K_{\Lambda}$
is elliptic. 

Recall that $\zeta$ is transcendental over $K_{\Lambda}$, consider
the polynomial ring $K_{\Lambda}[\zeta]$, and write $K_{\Lambda}[\zeta]_{<d}$
for the polynomials of degree $<d$ in $\zeta$.
\begin{lem}
\label{lem: injectivity =00005Cphi-a}Let $a\in\mathbb{C}.$ If $a\ne q^{r},$
the operator $(\phi-a)$ is injective on $F=\mathbb{C}((z))$, while
$\ker(\phi-q^{r})=\mathbb{C}z^{r}.$ A-fortiori, the same applies
to $S_{\Lambda}$. If $a\ne1$, the operator $(\phi-a)$ is injective
on $K_{\Lambda}[\zeta]$, while $\ker(\phi-1)=\mathbb{C}.$
\end{lem}

\begin{proof}
The first statement is clear, since $\phi(\sum a_{n}z^{n})=\sum a_{n}q^{n}z^{n}.$
The second follows from the fact that for $r\ne0,$ $z^{r}\notin K_{\Lambda}[\zeta]$.
See Lemma~\ref{lem:alg indep of z and zeta}.
\end{proof}
\begin{lem}
\label{lem:Lemma A.7}Let $f=\sum_{i=0}^{d}f_{i}\zeta^{i}\in K_{\Lambda}[\zeta]$
($f_{i}\in K_{\Lambda},$ $f_{d}\ne0$). Let $a\in\mathbb{C}^{\times}.$
Then:
\begin{enumerate}[i)]
\item  If $a\ne q^{d},$ or $a=q^{d}$ and $f_{d}\notin\mathbb{C}$,
then $(\phi-a)(f)$ has degree $d$ as a polynomial in $\zeta$ and
the coefficient of $\zeta^{d}$ is $q^{d}\phi(f_{d})-af_{d}.$

\item  If $d=0$, $a=1$ and $f_{0}\in\mathbb{C}$ then $(\phi-1)(f)=0.$

\item  If $d\ge1,$ $a=q^{d}$ and $f_{d}\in\mathbb{C}$ then $(\phi-q^{d})(f)$
has degree $d-1$ in $\zeta$ and the coefficient of $\zeta^{d-1}$
is
\[
df_{d}q^{d-1}f_{\zeta}+q^{d-1}(\phi-q)(f_{d-1}).
\]
\end{enumerate}

\end{lem}

\begin{proof}
We have $\phi(f)=\sum_{i=0}^{d}\phi(f_{i})(q\zeta+f_{\zeta})^{i}$,
so the coefficient of $\zeta^{d}$ in $(\phi-a)(f)$ is $q^{d}\phi(f_{d})-af_{d}.$
If $f_{d}$ is non-constant, it must have a pole at some $z_{0}\ne0.$
If we take $z_{0}$ to be a non-zero pole with minimal absolute value,
then $q^{d}\phi(f_{d})-af_{d}$ has a pole at $z_{0}/q,$ and in particular
can not vanish. If $f_{d}$ is constant, the coefficient of $\zeta^{d}$
vanishes only if $a=q^{d}.$ This proves i), and ii) is obvious.
In case  iii) the coefficient of $\zeta^{d}$ vanishes and the next
coefficient, of $\zeta^{d-1}$, comes out as stated from the same
computation. If it vanished, there would be an elliptic function $h$
such that
\[
(\phi-q)(\zeta)=f_{\zeta}=(\phi-q)(h).
\]
This contradicts the injectivity of $(\phi-q)$ on $K_{\Lambda}[\zeta]$.
\end{proof}
The next Lemma says that a function of the form $u+r\zeta$ where
$u$ is elliptic, can not have a global meromorphic primitive, unless
$r=0$.
\begin{lem}
\label{lem:The eta lemma}If $u\in K_{\Lambda}$, $w$ is globally
meromorphic, $r\in\mathbb{C}$ and $u+r\zeta=w'$ then $r=0.$
\end{lem}

\begin{proof}
For $\omega\in\Lambda$ write $\chi(z,\omega)=w(z+\omega)-w(z)$.
Differentiating with respect to $z$ we get that $\chi'(z,\omega)=r\eta(\omega)$
where $\eta$ is the Legendre $\eta$-function of the lattice $\Lambda.$
Thus $\chi(z,\omega)=r\eta(\omega)z+\mu(\omega)$ for some $\mu(\omega)\in\mathbb{C}.$
We get
\[
\chi(z,\omega_{1}+\omega_{2})=\chi(z+\omega_{1},\omega_{2})+\chi(z,\omega_{1})=r\eta(\omega_{2})(z+\omega_{1})+r\eta(\omega_{1})z+\mu(\omega_{1})+\mu(\omega_{2}),
\]
so
\[
\mu(\omega_{1}+\omega_{2})=r\eta(\omega_{2})\omega_{1}+\mu(\omega_{1})+\mu(\omega_{2}).
\]
If $r\ne0$ this is absurd, since the left hand side is symmetric
in $\omega_{1}$ and $\omega_{2}$, but for an oriented basis $(\omega_{1},\omega_{2})$
of $\Lambda$ the Legendre relation gives $\eta(\omega_{2})\omega_{1}-\eta(\omega_{1})\omega_{2}=2\pi i,$
so the right hand side is not symmetric.
\end{proof}
\begin{lem}
\label{lem:A9}Consider $g,f\in K_{\Lambda}[\zeta]$ and $a,\gamma\in\mathbb{C}$
such that
\[
g'=(q\phi-a)(f)+\gamma.
\]
Then there exists a $u\in K_{\Lambda}[\zeta]$ and $\beta\in\mathbb{C}$
such that
\[
g=(\phi-a)(u)+\beta.
\]
Furthermore, if $a\ne1,$ we may take $\beta=0.$
\end{lem}

\begin{proof}
The last statement is clear, because if $a\ne1,$ the operator $(\phi-a)$
is surjective on the constants, so $\beta=(\phi-a)(u_{0})$ for some
$u_{0}\in\mathbb{C},$ and replacing $u$ by $u+u_{0}$ we may assume
$\beta=0$.

We shall prove the lemma by induction on $\ell,$ the degree of $g$
as a polynomial in $\zeta$. Write
\[
g=\sum g_{i}\zeta^{i},\,\,\,f=\sum f_{i}\zeta^{i}\,\,\,(g_{i},f_{i}\in K_{\Lambda}).
\]

\bigskip{}

\textbf{Case $\ell=0$. }Assume that $g\in K_{\Lambda}$ and $g'=(q\phi-a)(f)+\gamma$
for $f\in K_{\Lambda}[\zeta],$$\gamma\in\mathbb{C}.$ By Lemma \ref{lem:Lemma A.7}
we are in one of the following two sub-cases:
\begin{itemize}
\item $f=f_{0}\in K_{\Lambda}$ and $g'=(q\phi-a)(f_{0})+\gamma,$
\item $a=q^{2},$ $f=f_{0}+f_{1}\zeta$ with $0\ne f_{1}\in\mathbb{C}$
and $g'=q(\phi-q)(f_{0})+qf_{1}f_{\zeta}+\gamma.$
\end{itemize}
Quite generally, if $\xi\in\mathbb{C}$ and $h(z)$ is meromorphic
at $\xi,$ let us say that $\xi$ is a \emph{residual point} of $h(z)$
if
\[
Res_{\xi}h(z)dz\ne0.
\]
If $\xi$ is a residual point of $h(z)$, then $\xi/q$ is a residual
point of $h(qz)$. A globally meromorphic function $h$ admits a primitive,
\emph{i.e.},. $h=w'$ for some globally meromorphic $w$, if and only if $h$
has no residual points.

In the first sub-case, we claim that $f_{0}$ has no residual points.
Otherwise, thanks to the periodicity of $f_{0},$ there would be a
residual point $\xi\ne0$, and we can take it to be of minimal absolute
value. The function $(q\phi-a)(f_{0})+\gamma$ then has $\xi/q$ as
a residual point, conradicting the fact that it has a meromorphic
primitive $g$.

It follows that there exists a globally meromorphic primitive $w$
with $w'=f_{0}$. Since $f_{0}$ is $\Lambda$-periodic,
\[
\chi(\omega)=w(z+\omega)-w(z)
\]
($\omega\in\Lambda)$ does not depend on $z,$ and is a homomorphism
$\Lambda\to\mathbb{C}.$ But every such homomorphism is supplied by
a linear combination of $z$ and $\zeta$ as well. It follows that
for some $r,s\in\mathbb{C}$ and $u\in K_{\Lambda}$ we have
\[
w=u+r\zeta+sz.
\]
Integrating the given expression for $g'$ we find that $g=(\phi-a)(w)+\gamma z+\beta$
for some $\beta\in\mathbb{C},$ or
\[
g=(\phi-a)(u+r\zeta)+(sq-sa+\gamma)z+\beta.
\]
As $g$ is elliptic, the term with $z$ must vanish (and in fact $r=0$
unless $a=q$), proving the Lemma in this case.

We claim that the second sub-case \emph{never occurs}. We can assume,
dividing $g$ by $qf_{1}$ to ease notation, that
\[
g'=(\phi-q)(f_{0}+\zeta)+\gamma.
\]
Now, $f_{0}+\zeta$ is not $\Lambda$-periodic, but its polar part
is, so had there been a    residual point  $\xi$ for it, so would be
$\xi+\omega$ for any $\omega\in\Lambda,$ and we could assume that
$\xi\ne0.$ The same argument as before shows that $f_{0}+\zeta$
has no residual points at all, hence
\[
f_{0}+\zeta=w'
\]
for some meromorphic $w.$ This contradicts  Lemma~\ref{lem:The eta lemma}.

\bigskip{}

\textbf{Induction step.} Let $\ell\ge1$, and assume that the lemma
had been proved up to degree $\ell-1$.

\textbf{Case $g_{\ell}\in\mathbb{C}$ and $a\ne q^{\ell}.$ }In this
case\textbf{
\[
g=(\phi-a)(v)+\widetilde{g}
\]
}where $v=g_{\ell}(q^{\ell}-a)^{-1}\zeta^{\ell}\in K_{\Lambda}[\zeta]$
and $\widetilde{g}\in K_{\Lambda}[\zeta]_{<\ell}$ (has degree smaller
than $\ell$ in $\zeta$). Clearly $\widetilde{g}$ satisfies the
assumption of the lemma on $\widetilde{g}'$, so by the induction
hypothesis is of the form $\widetilde{g}=(\phi-a)(\widetilde{u})+\gamma,$
with $\widetilde{u}\in K_{\Lambda}[\zeta]$ and $\gamma\in\mathbb{C}$.
It follows that so is $g,$ with $u=\widetilde{u}+v.$

\bigskip{}

\textbf{Case $g_{\ell}\in\mathbb{C}$ and $a=q^{\ell}.$ }We claim
that this case \emph{does not occur}. We have
\[
g'=(g_{\ell}\ell\zeta'+g_{\ell-1}')\zeta^{\ell-1}\mod K_{\Lambda}[\zeta]_{<\ell-1}
\]
and $f$ must be, according to Lemma~\ref{lem:Lemma A.7}, of degree
$\ell-1.$ Comparing coefficients of $\zeta^{\ell-1}$ we arrive at
\[
g_{\ell}\ell\zeta'+g_{\ell-1}'=q^{\ell}(\phi-1)(f_{\ell-1})+\gamma_{1}
\]
where $\gamma_{1}=\gamma$ if $\ell=1$ and $\gamma_{1}=0$ otherwise.
By the same argument as above, on   residual points, $f_{\ell-1}=w'$
for some meromorphic function $w$, and $w=u+r\zeta+sz$ for $u\in K_{\Lambda},\,\,\,r,s\in\mathbb{C}.$
Integrating we get that
\[
g_{\ell}\ell\zeta+g_{\ell-1}=q^{\ell-1}(\phi-q)(w)+\gamma_{1}z+\beta=q^{\ell-1}(\phi-q)(u)+rq^{\ell-1}f_{\zeta}+\gamma_{1}z+\beta
\]
($\beta\in\mathbb{C}).$ This contradicts Lemma \ref{lem:alg indep of z and zeta},
as $\zeta$ does not show up on the right hand side.

\bigskip{}

\textbf{Case $g_{\ell}\notin\mathbb{C}$ and $a\ne q^{\ell+1},q^{\ell+2}.$
}In this case $g'\equiv g_{\ell}'\zeta^{\ell}\mod K_{\Lambda}[\zeta]_{<\ell}.$
Since $a\ne q^{\ell+2},$ Lemma~\ref{lem:Lemma A.7} shows that $f$
must be of degree $\ell$ in $\zeta$ and, comparing coefficients
of $\zeta^{\ell}$,
\[
g_{\ell}'=(q^{\ell+1}\phi-a)(f_{\ell}).
\]
As before, this implies that $f_{\ell}=w'$ for some meromorphic $w$,
and that $w=v+r\zeta+sz$ for $v\in K_{\Lambda},$ $r,s\in\mathbb{C}.$
Integrating this gives
\[
g_{\ell}=(q^{\ell}\phi-a)(v+r\zeta+sz)+\alpha
\]
($\alpha\in\mathbb{C})$. As $a\ne q^{\ell+1}$ we must have $r=s=0,$
since the left hand side is elliptic, $\phi(z)=qz$ and $\phi(\zeta)=q\zeta\mod K_{\Lambda}.$
Note $(\phi-a)(v\zeta^{\ell})\equiv(q^{\ell}\phi-a)(v)\zeta^{\ell}\mod K_{\Lambda}[\zeta]_{<\ell}$.
Letting 
\[
\widetilde{g}=g-(\phi-a)(v\zeta^{\ell})=\sum_{i=0}^{\ell}\widetilde{g}_{i}\zeta^{i}
\]
we see that (a) $\widetilde{g}'$ satisfies the hypothesis of the
Lemma, with $\widetilde{f}=f-(v\zeta^{\ell})'$ (b) $\widetilde{g}_{\ell}=\alpha\in\mathbb{C}$,
hence we are in a case studied before. Thus $\widetilde{g}=(\phi-a)(\widetilde{u})+\beta$
(for some $\widetilde{u}\in K_{\Lambda}[\zeta],$ $\beta\in\mathbb{C})$,
and $u=\widetilde{u}+v\zeta^{\ell}$ solves our problem.

\bigskip{}

\textbf{Case $g_{\ell}\notin\mathbb{C}$ and $a=q^{\ell+1}.$ }Up
to the point where
\[
g_{\ell}=(q^{\ell}\phi-a)(v+r\zeta+sz)+\alpha
\]
the proof is as in the previous case. Furthermore, we may assume that
$s=0$ since $(q^{\ell}\phi-a)(z)=0,$ and that $\alpha=0$ because
it is of the form $(q^{\ell}\phi-a)(\alpha_{0})$ for $\alpha_{0}\in\mathbb{C}$
so can be ``swallowed'' in $v.$ Dividing by $q^{\ell}$ we may
write
\[
g_{\ell}=(\phi-q)(h+e\zeta)
\]
with $h\in K_{\Lambda}$ and $e\in\mathbb{C}.$ A direct computation
shows that
\[
g_{\ell}\zeta^{\ell}\equiv q^{-\ell}(\phi-q^{\ell+1})(h\zeta^{\ell}+\frac{e\zeta^{\ell+1}}{\ell+1})\mod K_{\Lambda}[\zeta]_{<\ell}.
\]
As before, letting $\widetilde{g}=g-q^{-\ell}(\phi-q^{\ell+1})(h\zeta^{\ell}+\frac{e\zeta^{\ell+1}}{\ell+1})\in K_{\Lambda}[\zeta]_{<\ell}$,
this $\widetilde{g}$ satisfies the hypothesis of the Lemma, so by
the induction hypothesis is of the form $(\phi-q^{\ell+1})(\widetilde{u})$,
with some $\widetilde{u}\in K_{\Lambda}[\zeta].$ It follows that
$g=(\phi-q^{\ell+1})(u)$ for an appropriate $u\in K_{\Lambda}[\zeta].$

\bigskip{}

\textbf{Case $g_{\ell}\notin\mathbb{C}$ and $a=q^{\ell+2}.$ }By
Lemma \ref{lem:Lemma A.7} $f$ has degree $\ell$ or $\ell+1$ in
$\zeta,$ and in the second case $f_{\ell+1}\in\mathbb{C}.$ Assume
first that we are in this second case, so $0\ne f_{\ell+1}\in\mathbb{C}.$
Comparing coefficients of $\zeta^{\ell}$ in $g'$ we get
\[
g'_{\ell}=q^{\ell+1}(\phi-q)\{f_{\ell+1}(\ell+1)\zeta+f_{\ell}\}.
\]
As argued before, $h=f_{\ell+1}(\ell+1)\zeta+f_{\ell}$, though not
periodic, has a  periodic polar part, so if $\xi$ is a  residual point
for it, so is $\xi+\omega$ for $\omega\in\Lambda.$ Had there been
a residual point for $h$, we could therefore take such a point $0\ne\xi$
of minimal absolute value, and then $\xi/q$ would be a residual point
for the right hand side. We conclude that $h$ has no residual points
at all, hence has a meromorphic primitive. This contradicts, however,
Lemma \ref{lem:The eta lemma}. Thus we can not have $f_{\ell+1}\ne0,$
and $f$ has degree $\ell.$

It follows that
\[
g'_{\ell}=q^{\ell+1}(\phi-q)(f_{\ell}).
\]
As before, this implies that $f_{\ell}$ has no residual points, hence
admits a primitive: $f_{\ell}=w'$, and $w=v+r\zeta+sz$ with $v\in K_{\Lambda},$
$r,s\in\mathbb{C}.$ Thus
\[
g_{\ell}=q^{\ell}(\phi-q^{2})(v+r\zeta+sz)+\beta,
\]
($\beta\in\mathbb{C})$, but we can ``swallow'' $\beta$ in $v$,
so we may assume $\beta=0.$ As the left hand side is elliptic, this
forces $r=s=0,$ and $g_{\ell}=q^{\ell}(\phi-q^{2})(v).$ This yields
\[
g_{\ell}\zeta^{\ell}\equiv(\phi-q^{\ell+2})(v\zeta^{\ell})\mod K_{\Lambda}[\zeta]_{<\ell}.
\]
Replacing $g$ by $\widetilde{g}=g-(\phi-q^{\ell+2})(v\zeta^{\ell})\in K_{\Lambda}[\zeta]_{<\ell}$
we may apply induction to conclude that $\widetilde{g}=(\phi-q^{\ell+2})(\widetilde{u}),$
for $\widetilde{u}\in K_{\Lambda}[\zeta],$ hence $g=(\phi-q^{\ell+2})(u)$
with $u=\widetilde{u}+v\zeta^{\ell}.$ This concludes the proof of
the Lemma.
\end{proof}
We can now finish the proof of  Proposition~\ref{prop:technical_elliptic}.
\begin{proof}
Set
\[
g=\sum_{i=0}^{M}g_{i}z^{i},\,\,\,f=\sum_{i=0}^{M}f_{i}z^{i},\,\,\,p=\sum_{i=0}^{M}p_{i}z^{i}
\]
where $g_{i},f_{i}\in K_{\Lambda}[\zeta]$ and $p_{i}\in\mathbb{C}.$
For $i\notin[0,M]$ let $g_{i}=f_{i}=p_{i}=0.$ Equating the coefficients
of $z^{i}$ in $(\ref{eq:delta-phi})$ we get
\begin{equation}
g_{i-1}'+(i-c)g_{i}=(q^{i}\phi-a)(f_{i})+p_{i}.\label{eq:recursion}
\end{equation}
For $i=M+1$ this gives $g_{M}\in\mathbb{C},$ so
\[
g_{M}=(q^{M}\phi-a)(u_{M})+\beta_{M}
\]
with $u_{M}\in K_{\Lambda}[\zeta]$ and $\beta_{M}\in\mathbb{C}$
trivially (e.g. $u_{M}=0,$ $\beta_{M}=g_{M}).$ 

For $i=M,$ substituting the latter expression for $g_{M}$ in $(\ref{eq:recursion})$
we find that
\[
g_{M-1}'=(q^{M}\phi-a)(\widetilde{f}_{M-1})+\gamma_{M-1}
\]
for some $\widetilde{f}_{M-1}\in K_{\Lambda}[\zeta]$ and $\gamma_{M-1}\in\mathbb{C}.$
Lemma \ref{lem:A9} ensures that
\[
g_{M-1}=(q^{M-1}\phi-a)(u_{M-1})+\beta_{M-1}
\]
with $u_{M-1}\in K_{\Lambda}[\zeta]$ and $\beta_{M-1}\in\mathbb{C}.$
Iterating, using the recursion formula $(\ref{eq:recursion})$ and
Lemma \ref{lem:A9}, we solve succesively for
\[
g_{i}=(q^{i}\phi-a)(u_{i})+\beta_{i},
\]
with $u_{i}\in K_{\Lambda}[\zeta]$ and $\beta_{i}\in\mathbb{C},$
giving the desired equation
\[
g=(\phi-a)\left(\sum_{i=0}^{M}u_{i}z^{i}\right)+\sum_{i=0}^{M}\beta_{i}z^{i}.
\]
The final reduction of $\widetilde{p}=\sum_{i=0}^{M}\beta_{i}z^{i}$
to a monomial $dz^{r}$ (if $a=q^{r})$ or 0 (otherwise), by a suitable
modification of the $u_{i}$, is obvious, as ($\phi-a)(z^{i})=(q^{i}-a)z^{i}$. 
\end{proof}

\section{Appendix B: Galois invariants\label{subsec:Galois-invariants}}

We want to prove a technical result, for which we could not find a
reference.

Let $K$ be a ($\phi,\delta)$-field with $K^{\phi}=C$ algebraically
closed. Let $\widetilde{C}$ be a differential closure of $C$. It
is unique up to isomorphism. Let
\[
K^{\dagger}=\widetilde{C}\otimes_{C}K
\]
(a domain, but in general not a field).

Let $\phi(y)=Ay$
be a  difference system with  $A\in \GL_{n}(K)$ and let $\mathcal{L}=K\langle U \rangle_{\delta}$ be a PPV extension for this system over $K$. We   assume that $\mathcal{L}$ is a field. We consider 
\[
R=K[U,\det(U)^{-1}],\,\,\,\mathcal{R}=K\{U,\det(U)^{-1}\}_{\delta}
\]
the PV ring and the $\delta$-parametrized PV ring inside $\mathcal{L}$.  Note that $R^{\phi}=\mathcal{R}^{\phi}=C$. 
Let
\[
R^{\dagger}=\widetilde{C}\otimes_{C}R=K^{\dagger}[U,\det(U)^{-1}],\,\,\,\mathcal{R}^{\dagger}=\widetilde{C}\otimes_{C}\mathcal{R}=K^{\dagger}\{U,\det(U)^{-1}\}_{\delta}.
\]

Let $
\mathcal{\widetilde{G}}=\mathrm{Aut}_{\phi,\delta}(\mathcal{R}^{\dagger}/K^{\dagger}).
$
This is the $\delta$-parametrized Galois group over $\widetilde{C}$
and we view it as a subgroup of $\GL_{n}(\widetilde{C})$. If $\mathcal{G}$
is the $\delta$-parametrized Galois group scheme attached  to $\mathcal{R}$, then $\widetilde{\mathcal{G}}=\mathcal{G}(\widetilde{C}).$
\begin{prop}\label{prop:refinmentparamgalcorrespondence}
If $a\in\mathcal{R}^{\dagger}-K^{\dagger},$ then there exists a $g\in\widetilde{\mathcal{G}}$
with $g(a)\ne a.$
\end{prop}

The Proposition is usually phrased when $K^{\dagger}$ is replaced
by $Frac(K^{\dagger})$ and $\mathcal{R}^{\dagger}$ by $Frac(\mathcal{R}^{\dagger}),$
the $\delta$-parametrized PV extension. However, it is not a-priori
clear that $\mathcal{R}^{\dagger}\cap Frac(K^{\dagger})=K^{\dagger}.$
\begin{proof}
We have adapted  slightly the arguments of \cite[Lemma 3.1]{DVHW} to our context.  Consider
\[
d=a\otimes1-1\otimes a\in\mathcal{R}^{\dagger}\otimes_{K^{\dagger}}\mathcal{R}^{\dagger}=\widetilde{C}\otimes_{C}(\mathcal{R}\otimes_{K}\mathcal{R}).
\]
Since $a\notin K^{\dagger},$ $d\ne0.$ This is clear if $\mathcal{R}^{\dagger}$
and $K^{\dagger}$ are replaced by $\mathcal{R}$ and $K$, since
$K$ is a field. Expanding $a=\sum\omega_{\alpha}\otimes a_{\alpha}$
with $\{\omega_{\alpha}\}$ a basis of $\widetilde{C}$ over $C$,
we see that $d=\sum\omega_{\alpha}\otimes(a_{\alpha}\otimes1-1\otimes a_{\alpha}).$
Since $a \notin K^{\dagger}$, there is some   $a_{\alpha_0}\in\mathcal{R}-K$ so that  $a_{\alpha_0}\otimes1-1\otimes a_{\alpha_0}$ is a nonzero element in $\mathcal{R}\otimes_K \mathcal{R}$.

The torsor theorem  is equivalent to the fact that the inclusion $\left(\mathcal{R} \otimes_K  \mathcal{R} \right)^{\phi} \hookrightarrow \mathcal{R} \otimes_K  \mathcal{R}$ extends to a $\mathcal{R}$-$(\phi,\delta)$-isomorphism 
\[  \Theta:  \left(\mathcal{R} \otimes_K  \mathcal{R} \right)^{\phi}  \otimes_{C}   \mathcal{R}  \rightarrow \mathcal{R} \otimes_K  \mathcal{R}.\]
The $C$-$\delta$-algebra $C\{ \mathcal{G}\}=\left(\mathcal{R} \otimes_K  \mathcal{R} \right)^{\phi}$ represents the functor $\mathcal{G}$ (see Proposition 6.18 in  \cite{H-S08} for the case where $C$ is differentially closed and the discussion after its proof to see how it can be adapted to our situation where ${Frac(\mathcal{R})}^{\phi}= C$.) An element $g \in \mathcal{G}(\widetilde{C})$ thereby corresponds to a $C$-$\delta$-algebra morphism $\psi$ from $C\{\mathcal{G}\}$ to $\widetilde{C}$ whose kernel is a prime $\delta$-ideal $\mathfrak{m}_{g}$.  Its action on an element $1 \otimes b \in\widetilde{C} \otimes_C \mathcal{R}$ is given by $ (\psi \otimes id) \circ \Theta^{-1} (b \otimes 1)$ (see for instance the discussion in \cite[Lemma 3.1]{DVHW} in the analogous context where the roles of $\delta$ and $\phi$ are interchanged). A nonzero element $1\otimes b $ of $ \widetilde{C} \otimes_C \mathcal{R}$  is $\mathcal{G}(\widetilde{C})$-invariant if $ \Theta^{-1}(b \otimes 1 -1 \otimes b)$ belongs to $\cap_{g \in \mathcal{G}(\widetilde{C})} \mathfrak{m}_{g} \otimes_C \mathcal{R}$. Now, we claim that  $\cap_{g \in \mathcal{G}(\widetilde{C})} \mathfrak{m}_{g}$ is the zero ideal in $C\{\mathcal{G}\}$. Indeed, since $C\{G \}$ is reduced, differentially finitely generated over $C$ and $\widetilde{C}$ is differentially closed, this is a direct consequence of the differential Nullstellensatz (see \cite{M-MTDF}). To conclude, we have proved that  if a nonzero element $1\otimes b \in \mathcal{R}^{\dagger}$ is invariant under $\mathcal{G}(\widetilde{C})$ then $b\otimes 1 -1 \otimes b =0$ in $\mathcal{R}\otimes_{K} \mathcal{R}$.

Therefore, there exists an element $g$ in $\mathcal{G}(\widetilde{C})$ such that $g(\omega_{\alpha_0} \otimes a_{\alpha_0}) \neq \omega_{\alpha_0} \otimes a_{\alpha_0}$. This   yields $g(a) \neq a$ and concludes the proof.
\end{proof}

\end{document}